\documentclass[10pt]{amsart}
\usepackage[utf8]{inputenc}

\usepackage{amsmath}
\usepackage{amsthm}
\usepackage{accents}
\usepackage{amssymb}
\usepackage{graphicx}
\usepackage{color}
\usepackage{enumerate}

\newcommand{\ff}{h}
\newcommand{\nth}{\tfrac{1}{n-1}}

\newcommand{\Riem}{\Rm}

\newcommand{\Scal}{\operatorname {R}}
\newcommand{\Ric}{\operatorname{Rc}}
\newcommand{\Rict}{\overline{\Ric}}
\newcommand{\ra}{\rangle}
\newcommand{\la}{\langle}
\newcommand{\half}{\tfrac{1}{2}}
\newcommand{\B}{\mathcal{B}}
\newcommand{\C}{\mathcal{C}}
\newcommand{\CH}{\mathcal{H}}
\newcommand{\area}{\mathcal{H}^{n-1}}

\newcommand{\dmu}{d \area}

\newcommand{\tr}{\operatorname{tr}}
\newcommand{\del}{\partial}
\newcommand{\CL}{\mathcal{L}}

\newcommand{\R}{\mathbb{R}}
\renewcommand{\S}{\mathbb{S}}
\newcommand{\graph}{\operatorname{graph}}
\newcommand{\Rm}{\operatorname{Rm}}

\newcommand{\supp}{\operatorname{supp}}

\newcommand{\gt}{\overline{g}}
\newcommand{\tf}{\mathring{h}}
\newcommand{\Lapt}{\overline{\Delta}}
\newcommand{\nabt}{\overline{\nabla}}

\newcommand{\id}{\operatorname{id}}

\renewcommand{\div}{\operatorname{div}}

\renewcommand {\H} {\mathcal{H}}

\newtheorem{theorem}{Theorem}[section]
\newtheorem{lemma}[theorem]{Lemma}
\newtheorem{definition}[theorem]{Definition}
\newtheorem{corollary}[theorem]{Corollary}
\newtheorem{proposition}[theorem]{Proposition}
\theoremstyle{remark}
\newtheorem{remark}[theorem]{Remark}

\title[Unique isoperimetric foliations]{Unique isoperimetric foliations of asymptotically flat manifolds in all dimensions}

\author{Michael Eichmair \and Jan Metzger}
\address{Michael Eichmair, ETH Z\"urich, Departement Mathematik, 8092 Z\"urich, Switzerland}
\email{michael.eichmair@math.ethz.ch}

\address{Jan Metzger, Universit\"at Potsdam, Institut f\"ur Mathematik, Am
Neuen Palais 10, 14469 Potsdam, Germany}
\email{jan.metzger@uni-potsdam.de}

\thanks{}

\date{\today}

\begin{document}

\maketitle


\section{Introduction}

The question of isoperimetry \emph{What is the largest amount of volume that can be enclosed by a given amount of area?} can be traced back to antiquity.\footnote{We refer the reader to \cite{Dido} for a beautiful collection of materials on the history of the isoperimetric problem.} The first mathematically rigorous results are as recent as the nineteenth century. The question of isoperimetry and its close relative, the analysis of minimal surfaces, are two of the model problems of the geometric calculus of variations.

The list of geometries where an explicit answer to the question of isoperimetry is available is short. We provide an overview of available results in Appendix \ref{sec:overview}. In \cite{Eichmair-Metzger:2010} and this paper, we extend this list by a class of Riemannian manifolds $(M, g)$ for which we describe all large isoperimetric regions completely. 

We refer the reader to Section \ref{sec:definitions-notation} for the precise definitions of all terms in the statement of our first main theorem:  

\begin{theorem} \label{thm:main} Let $(M, g)$ be an $n$-dimensional initial data set with $n \geq 3$ that is $\C^0$-asymptotic to Schwarzschild of mass $m>0$. There exists $V_0>0$ such that for every $V \geq V_0$ the infimum in 
  \begin{equation} \label{eqn:mainiso}
    \begin{split}
      \inf \{\H^{n-1}_g(\partial \Omega) : \Omega \subset M 
      \text{ is a smooth region with } \CL^n_g(\Omega) = V \}
    \end{split}
  \end{equation} 
is achieved by a smooth isoperimetric region $\Omega_V \subset M$. The boundary of $\Omega_V$ is close to a centered coordinate sphere $S_r$ where $r$ is such that $\CL^n_g(B_r) = V = \CL^n_g(\Omega_V)$. If $(M, g)$ is $\C^2$-asymptotic to Schwarzschild of mass $m>0$, then $V_0>0$ can be chosen such that for every $V \geq V_0$, $\Omega_V$ is the unique minimizer of \eqref{eqn:mainiso}, and such that the hypersurfaces $\{\partial \Omega_V\} _{V \geq V_0}$ foliate $M \setminus \Omega_0$ smoothly. If $(M, g)$ is also asymptotically even, then the centers of mass of the boundaries $\partial \Omega_V$ converge to the center of mass of $(M, g)$ as $V \to \infty$. 
\end{theorem}

Theorem \ref{thm:main} follows from Theorems \ref{thm:existence} and \ref{thm:center_of_mass}. The class of Riemannian manifolds $(M, g)$ to which Theorem \ref{thm:main} applies appears naturally in mathematical relativity as initial data for the Einstein equations. It also appears naturally in conformal geometry: If $(\bar M, \bar g)$ is a closed Riemannian manifold of positive Yamabe type and if either $3 \leq n \leq 5$ or if $(\bar M, \bar g)$ is locally conformally flat, then $(\bar M \setminus \{p\}, G^{\frac{4}{n-2}} \bar g)$ lies in this class. Here, $p \in \bar M$ and $G$ is the Green's function with pole at $p$ of the conformal Laplace operator of $g$, cf. Theorem V.3.6 in \cite{Schoen-Yau:1997} and Propositions 3.3 and 4.4 in \cite{Schoen-Yau:1988}.  

The critical points for the isoperimetric problem are exactly the constant mean curvature surfaces. Stable critical points are volume preserving stable constant mean curvature surfaces. The study of such surfaces also has a rich (if relatively recent) history. We mention in particular Alexandrov's theorem which shows that closed constant mean curvature surfaces in Euclidean space are round spheres. In their seminal paper \cite{Huisken-Yau:1996}, G. Huisken and S.-T. Yau proved that the complement of a bounded set of a three dimensional initial data set $(M, g)$ that is $\C^4$-asymptotic to Schwarzschild of mass $m>0$ is foliated by strictly volume preserving stable constant mean curvature spheres. Moreover, the leaves of this foliation are the unique volume preserving stable constant mean curvature spheres of their mean curvature within a large class of surfaces, including all nearby ones. This uniqueness has been extended to a larger class of surfaces in important work by J. Qing and G. Tian \cite{Qing-Tian:2007}. G. Huisken and S.-T. Yau have also shown in \cite{Huisken-Yau:1996} that the centers of mass of their surfaces have a limit, the ``Huisken-Yau geometric center of mass". H. Bray conjectured in \cite{Bray:1998} that the surfaces found in \cite{Huisken-Yau:1996} are in fact isoperimetric. We have confirmed this conjecture in \cite{Eichmair-Metzger:2010} by establishing an effective volume comparison result for initial data sets that are $\C^0$-asymptotic to Schwarzschild of mass $m>0$, building on H. Bray's characterization \cite{Bray:1998} of the isoperimetric regions in the exact spatial Schwarzschild geometry. We refer the reader to the introductions of \cite{Eichmair-Metzger:2010, Eichmair-Metzger:2011a} for a more extensive discussion including the physical significance of these results and further references. 

Theorem \ref{thm:main} here extends these results in several ways:
\begin{enumerate} [(i)]
\item Our result holds in all dimensions. The existence of a foliation by volume preserving stable constant mean curvature surfaces when $(M, g)$ is $\C^4$-asymptotic to Schwarzschild of mass $m\neq0$ has been shown by R. Ye in \cite{Ye:1996} in all dimensions. The proofs of the uniqueness results for large volume preserving stable constant mean curvature surfaces in \cite{Huisken-Yau:1996, Qing-Tian:2007} depend delicately on some tools that are special to three dimensional initial data sets. 
\item The uniqueness of the leaves \emph{in the class of isoperimetric surfaces} in our result is \emph{global}. The uniqueness results of \cite{Huisken-Yau:1996, Qing-Tian:2007} apply only to surfaces that lie far in the asymptotic regime of the initial data set $(M, g)$ where the geometry is close to that of the exact spatial Schwarzschild geometry. An important ingredient in our proof is the recent characterization of closed constant mean curvature surfaces in the exact spatial Schwarzschild geometry by S. Brendle \cite{Brendle:2011}. 
\item Unlike \cite{Huisken-Yau:1996}, we only require that $(M, g)$ is $\C^2$-asymptotic to Schwarzschild of mass $m>0$. To accomplish this, we rely on strong a priori position estimates for large isoperimetric regions in initial data sets that are $\C^0$-asymptotic to Schwarzschild of mass $m>0$. These estimates come from our effective volume comparison result, see Theorems \ref{thm:effective_volume_comparison} and \ref{thm:centering}. In the case $n=3$, we also rely on an idea from \cite{Metzger:2007} which in turn depends on the effective version of Schur's lemma of C. De Lellis and S. M\"uller in \cite{DeLellis-Muller:2005}. 
\item When $n=3$ and $(M, g)$ is $\C^4$-asymptotic to Schwarzschild of mass $m>0$, the convergence of the centers of mass of the leaves of the foliation was established in \cite{Huisken-Yau:1996}. L.-H. Huang showed in \cite{Huang:2009, Huang:2010} that the Huisken-Yau geometric center of mass coincides with the usual center of mass \cite{Regge-Teitelboim:1974, Beig-OMurchadha:1987}. Our proof here, which works in all dimensions, uses ideas from \cite{Huang:2009, Huang:2010}, but it is both shorter and more elementary. In particular,  we do not rely on the delicate density theorem of \cite{Huang:2010}.  
\end{enumerate}
One key ingredient in our proof of Theorem \ref{thm:main} is an all-dimensional analogue of the effective volume comparison theorem for $3$-dimensional initial data sets $(M, g)$ that are $\C^0$-asymptotic to Schwarzschild of mass $m>0$ that was obtained by the authors in \cite{Eichmair-Metzger:2010}. This result is established in Section \ref{sec:volume_comparison}. 

In \cite{Ritore:2001B}, M. Ritor\'e has shown that in a complete Riemannian surface with non-negative curvature, isoperimetric regions $\Omega_V$ exist in $M$ for every volume $V>0$. Our second main result is the existence of large isoperimetric regions in \emph{arbitrary} $3$-dimensional initial data sets with non-negative scalar curvature: 

\begin{theorem} \label{thm:isomass_exist1} Assume that $(M, g)$ is a three dimensional initial data set that has non-negative scalar curvature. There exists a sequence of isoperimetric regions $\Omega_i\subset M$ with $\mathcal{L}^3_g(\Omega_i) \to \infty$.
\end{theorem}

The proof of Theorem \ref{thm:isomass_exist1}, which is given in Section \ref{sec:isomass}, is indirect and uses recent deep insights of G. Huisken's on the isoperimetric mass of initial data sets. Note that our theorem implies in particular that $(M, g)$ contains large volume preserving stable constant mean curvature surfaces. Using arguments as for example in \cite{Eichmair-Metzger:2011a} it follows that appropriate homothetic rescalings of these large isoperimetric regions to a fixed volume are close to coordinate balls. The existence of such surfaces in this generality seems to lie deep and out of reach of e.g. implicit-function type arguments. 

 \subsection*{Acknowledgements} We are very grateful to Hubert Bray,
 Simon Brendle, Gerhard Huisken, Manuel Ritor\'e, Brian White, and
 Shing-Tung Yau for useful conversations, encouragement, and
 support. We also thank the referees for their careful reading and valuable comments. Michael Eichmair gratefully acknowledges the support of  NSF
 grant DMS-0906038 and of SNF grant 2-77348-12. Also, Michael Eichmair
 wishes to express his sincere gratitude to Christina Buchmann,
 Katharina Halter, Madeleine Luethy, Alexandra Mandoki, Anna and Lisa
 Menet, Martine Verwey, Markus Weiss, and his wonderful colleagues in Group 6 at
 ETH for making him feel welcome and at home in Z\"urich right from
 the start.  


\section{Definitions and Notation}
\label{sec:definitions-notation}

\begin{definition}
  \label{def:initial_data_sets} Let $n \geq 3$. 
  An initial data set $(M, g)$ is a connected complete boundaryless $n$-dimensional
  Riemannian manifold such that there exists a bounded open set $U
  \subset M$ so that $M \setminus U
  \cong_x \R^n \setminus B_{\frac{1}{2}}(0)$, and such that in the
  coordinates induced by $x = (x_1, \ldots, x_n)$ we have that
  \begin{equation*}
    r |g_{ij} -
    \delta_{ij}| + r^2 |\partial_k g_{ij}| + r^{3} |\partial^2_{k l}
    g_{ij}| \leq C \text { for all } r \geq 1, 
  \end{equation*}
  where $r:= \sqrt {\sum_{i=1}^n x_i^2}$. 
  
  Given $m > 0$, $\gamma \in (0, 1]$, and an integer
  $k\geq 0$, we say that an initial data set is $\C^k$-asymptotic to
  Schwarzschild of mass $m>0$ at rate $\gamma$ if
  \begin{equation*}
    \sum_{l=0}^k  r^{n - 2 +\gamma + l}|\partial^l(g - g_m)_{ij}| \leq C \text { for all } r \geq 1, 
  \end{equation*}
  where $(g_m)_{ij} = (1 + \frac{m}{2|x|^{n-2}})^{\frac{4}{n-2}}
\delta_{ij}$. 

We say that an initial data set $(M, g)$ that is $\C^2$-asymptotic to Schwarzschild of mass $m>0$ at rate $\gamma$ is asymptotically even, if 
  \begin{equation*}
    r^{n+1+\gamma} |\Scal (x) - \Scal (-x)|
    \leq
    C \text{ for all } x \text{ with } |x| =r \geq \frac{1}{2}.  
  \end{equation*}
Here, $\Scal $ is the scalar curvature of $g$.
\end{definition}

We extend $r$ as a smooth regular function to the entire initial data
set $(M, g)$ such that $r(U) \subset [0, 1)$. We use $S_r$, $B_r$ to
denote the surface $\{ x \in M : |x| = r\}$ and the region $\{ x \in M
: |x| \leq r\}$ in $M$ respectively.  We will not distinguish between
the end $M \setminus U$ of $M$ and its image $\R^n \setminus
B_{\frac{1}{2}} (0)$ under $x$.

If $\Omega \subset M$ is Borel and has locally finite perimeter, then its reduced boundary in
$(M, g)$ is denoted by $\partial^* \Omega$.
 
\begin{definition}
  \label{defn:isoprofile}
  The isoperimetric area function $A_g : [0, \infty) \to [0, \infty)$
  is defined by
  \begin{equation*}
    \label{eqn:isoperimetric_area_function}
    \begin{split}
      A_g(V) &:=  \inf \{\H^{n-1}_g(\partial^* \Omega) : \Omega \subset M
      \text{ is Borel, has finite perimeter, and } \CL^n_g(\Omega) = V \}.
    \end{split}
  \end{equation*} 
  A Borel set $\Omega \subset M$ with finite perimeter such that
  $\CL^n_{g} (\Omega) = V $ and $A_g(V) = \H^{n-1}_{g} (\partial^*
  \Omega)$ is called an isoperimetric region of $(M, g)$ of volume
  $V$.
\end{definition}


\section{A refinement of Bray's isoperimetric comparison theorem for Schwarzschild in all dimensions}
\label{sec:volume_comparison}

Throughout this section we will use the notation for the spatial Schwarzschild manifold of mass $m>0$,  $$(\R^n \setminus \{0\}, g_m := \left( 1 + \frac{m}{2|x|^{n-2}}\right)^{\frac{4}{n-2}} \sum_{i=1}^n dx_i^2),$$ set forth in Appendix \ref{sec:geometrySchwarzschild}. 

In his thesis \cite{Bray:1998}, H. Bray has proven that the centered coordinate spheres in the three-dimensional Schwarzschild manifold of mass $m>0$ are isoperimetric. His argument also applies to compact perturbations of Schwarzschild provided the coordinate spheres are sufficiently large. The following proposition follows from straightforward modifications of H. Bray's original, three dimensional arguments in \cite{Bray:1998}, see also \cite[Section 3]{Eichmair-Metzger:2010}. The proof uses the analogue of the Hawking mass for rotationally symmetric Riemannian manifolds in higher dimensions, which we review in Appendix \ref{sec:Hawking}. See also \cite{Bray-Morgan:2002, Corvino-Gerek-Greenberg-Krummel:2007, Abedin-Corvino-Kapita-Wu:2009} for other extensions of the results in H. Bray's thesis. 

\begin{proposition} [Bray's volume preserving charts in higher dimension] Let $m>0$ and $r>r_h$ be given. There exist $\alpha \in (0, 1)$, $c >0$, $s_0 \in [0, c)$, $u_c \in \C^{1, 1} ((0, \infty))$, and $w_c \in \C^{1, 1} ((s_0, \infty))$ with the following properties: 
\begin{enumerate} [(a)]
\item The sphere $\{c\} \times \S^{n-1}$ in the metric cone $((0, \infty) \times \S^{n-1}, \alpha^{-2} ds^2 + \alpha^{\frac{2}{n-1}}s^2 g_{\S^{n-1}})$ has the same area and the same (positive) mean curvature as $S_r = S_r(0) \subset \R^n$ with respect to $g_m$. 
\item $u_c (s) = \alpha$ for $s \leq c$. $u_c$ is a smooth non-decreasing function on $(c, \infty)$ and $\lim_{s \to \infty} u_c(s) = 1$. The derivative of $u_c$ at $s = c$ exists and equals $0$.
\item $w_c \geq 1$, $w_c(s) \equiv 1$ for $s \geq c$, the derivative of $w_c$ at $s = c$ exists and is $0$, $\lim_{s \searrow s_0} w_c(s) = \infty$, and $w_c$ is smooth on $(s_0, c)$. 
\item Define the metric $g_m^c:= u_c^{-2} ds^2 + u_c^{\frac{2}{n-1}} s^2 g_{\S^{n-1}}$ on $(s_0, \infty) \times \S^{n-1}$. Then $((s_0, \infty) \times \S^{n-1}, w_c^{\frac{4}{n-2}} g_m^c)$ is isometric to $(\R^n \setminus \{0\}, g_m)$ under a rotationally invariant map that sends $\{c\} \times \S^{n-1}$ to $S_r(0)$. 
\item As quadratic forms, we have that $\alpha^2 g_m^c \leq  ds^2 + s^2 g_{\S^{n-1}} \leq u_c^{-\frac{2}{n-1}} g_{m}^c$ on $(s_0, \infty) \times \S^{n-1}$. 
\end{enumerate}
\end{proposition}

The additional observations about $\alpha, c$, and $u_c$ in the following proposition, in particular (\ref{u(tc)-u(c)}), are the key to making H. Bray's characterization of isoperimetric regions in Schwarzschild into an effective volume comparison theorem. We omit the proofs, which are simple adaptations of the arguments in \cite{Eichmair-Metzger:2010}. 

\begin{proposition} [\protect{Cf. \cite[Section 3]{Eichmair-Metzger:2010}}]\label{prop:propertiesu} We have that
\begin{enumerate} [(a)]
\item $c^n = r^n \left(1 +  \frac{2n-2}{n-2}\frac{m}{r^{n-2}}+O\left(\frac{1}{r^{2n-4}}\right)\right)$
\item $\alpha = 1 - \frac{n-1}{n} \frac{m}{r^{n-2}} + O\left(\frac{1}{r^{2n-4}}\right)$
\item The scalar curvature of the conical metric $((0, \infty) \times \S^{n-1}, \alpha^{-2} ds^2 + \alpha^{\frac{2}{n-1}}s^2 g_{\S^{n-1}})$ equals $(n-1)(n-2) s^{-2} (\alpha^{- \frac{2}{n-1}} - \alpha^2)$. 
\item The Schwarzschild volume between $S_r$ and the horizon $S_{r_h}$ is greater than the volume of $(0, c] \times \S^{n-1}$ with respect to the metric $\alpha^{-2} ds^2 + \alpha^{\frac{2}{n-1}}s^2 g_{\S^{n-1}}$. The difference is 
 $V_0
  =
  \frac{\omega_{n-1} r^n}{n} \frac{n-2}{2} \frac{m}{r^{n-2}} + O\left(\frac{1}{r^{n-4}}\right)
  =
  \frac{\omega_{n-1} c^n}{n} \frac{n-2}{2} \frac{m}{c^{n-2}} + O\left(\frac{1}{c^{n-4}}\right)$. 
\item \label{u(tc)-u(c)} Fix $\tau_0 >1$ and let $\tau \geq \tau_0$. Then 
\begin{eqnarray*}
u_c (\tau c ) - u_c (c) &=& u_c(\tau c) - \alpha \\ &=& \frac{(n-1)m}{2 n c^{n-2} \tau^n} (2 \tau^n - n \tau^2 + (n-2)) + O\left(\frac{1}{c^{2n-4}}\right) \\ &\geq& \delta \frac{m}{c^{n-2}} (1 - \frac{1}{\tau})^2
\end{eqnarray*}
provided that $c$ is sufficiently large (depending only on $m$ and $\tau_0$), and where $\delta>0$ is a constant depending only on $n$. 
\end{enumerate}
\end{proposition}

With these preparations, it is a simple matter to carry over the derivation of the effective volume comparison result \cite[Proposition 3.3]{Eichmair-Metzger:2010} to arbitrary dimensions. The result is based on the concept introduced in the following definition: 

\begin{definition} [\protect{Cf. \cite[Definition 3.2]{Eichmair-Metzger:2010}}]
  \label{def: off-center} Let $(M, g)$ be an initial data set that is
$\C^0$-asymptotic to Schwarzschild of mass $m >0$. Let $\Omega$ be a
bounded Borel set with finite perimeter in $(M, g)$. Given parameters $\tau >1$ and $\eta \in (0, 1)$ we say that such
a set $\Omega$ is $(\tau, \eta)$-off-center if
  \begin{enumerate} [(i)]
  \item $\CL^n_g (\Omega)$ is so large that there exists a coordinate
sphere $S_r = \partial B_r$ with $\CL^n_g(\Omega) = \CL^n_g(B_r)$ and $r \geq
1$, and if
  \item $\H^{n-1}_g (\partial^* \Omega \setminus B_{\tau r}) \geq \eta \H^{n-1}_g
(S_r)$.
  \end{enumerate}
\end{definition}

There are several measures of asymmetry in the literature that lead to effective versions of the classical isoperimetric inequality in Euclidean space, cf. Appendix \ref{sec:effectiveiso}. The following effective version of Bray's characterization of the isoperimetric regions in Schwarzschild is not a consequence of an effective isoperimetric inequality in Euclidean space; it depends on the positivity of the mass in a crucial way. 

\begin{proposition} [Effective Volume Comparison in Schwarzschild,
  cf. \protect{\cite[Proposition 3.3]{Eichmair-Metzger:2010}}] \label{prop:effectivevolume}
  Given $m>0$ and $(\tau, \eta) \in (1, \infty) \times (0, 1)$ there
  exists $V_0 >0$ so that the following holds: Let $V \geq V_0$ and
  let $r \geq r_h$ be such that $V= \CL^n_{g_m} (B_r \setminus
  B_{r_h})$ and let $\Omega \subset \R^n$ be a bounded Borel set with finite
  perimeter such that $B_{r_h} \subset \Omega$ and $\CL^n_{g_m}(\Omega
  \setminus B_{r_h}) = V$. If $\Omega$ is $(\tau, \eta)$-off-center,
  i.e. if $\H^{n-1}_{g_m} ( \partial^* \Omega \setminus B_{\tau r} )
  \geq \eta \mathcal{H}_{g_m} (S_r)$, then
\begin{equation} \label{eqn: comparison in Schwarzschild}
  \H^{n-1}_{g_m}(\partial^* \Omega) \geq \H^{n-1}_{g_m} (S_r) + c \eta m \left(1 -  \frac{1}{\tau}\right)^2  r.
\end{equation}
Here, $c>0$ is a constant that only depends on $n$.
\end{proposition}

The proof of the main theorem in this section below is literally the same as in \cite{Eichmair-Metzger:2010}, except for adapting various exponents throughout the proof. Since the modifications are delicate, we include the full argument.

\begin{theorem} [Cf. \protect{\cite[Theorem 3.4]{Eichmair-Metzger:2010}}] \label{thm:effective_volume_comparison} Let $(M, g)$ be an
initial data set that is $\C^0$-asymptotic to Schwarzschild of mass $m
>0$. For every tuple $(\tau, \eta) \in (1, \infty) \times (0, 1)$ and every
constant $\Theta >0$ there exists a constant $V_0 >0$ such that the
following holds: Given a bounded Borel set $\Omega$ with finite perimeter
in $(M, g)$  and with $\CL^n_g(\Omega) \geq V_0$ that is
$(\tau, \eta)$-off-center with $\H^{n-1}_g(\partial^* \Omega) ^{\frac{1}{n-1}}
\CL^n_g(\Omega)^{-\frac{1}{n}} \leq \Theta$ and such that $\H^{n-1}_g(\partial^*
\Omega \cap B_\sigma) \leq \Theta \sigma^{n-1}$ holds for all $\sigma \geq 1$,
one has $$\H^{n-1}_g(\partial^* \Omega) \geq\H^{n-1}_g(S_r) + c \eta m \left(1 -  \frac{1}{\tau} 
\right)^2 r$$ where $r\geq1$ is such that $\CL^n_g(B_r) = \CL^n_g(\Omega)$, and where $c >0$ is a constant that only depends on $n$.
\end{theorem}
\begin{proof}
For ease of exposition we only consider smooth regions $\Omega$. The
result for sets with finite perimeter follows from this by approximation.
We will use here that $\H^{n-1}_g (\partial \Omega) \to \infty$ as
$\CL^n_g(\Omega) \to \infty$, which follows from the isoperimetric
inequality in Lemma \ref{lem:crudeisoperimetricinequality}. Note also
that $\CL^n_g(\Omega)= \frac{ \omega_{n-1} r^n} {n} + O(r^{n-1})$. 

We break into several steps:

\begin{enumerate} [(a)]  
\item Let $\tilde \Omega := \Omega \cup B_{1} \subset M$. Let $\tilde
\Omega_m :=
\left( x (\Omega \setminus B_{1})  \cup B_1(0) \right) \setminus B_{r_h}(0)$ be the corresponding region in Schwarzschild.
\item Note that $\CL^n_g(\tilde \Omega) = \CL^n_g(\Omega) + O(1)$ and
$\H^{n-1}_g(\partial \tilde \Omega) = \H^{n-1}_g (\partial \Omega) + O(1)$.
Moreover, $\tilde \Omega$ satisfies $\H^{n-1}_g (\partial \tilde \Omega \cap
B_\sigma) \leq \tilde \Theta \sigma^{n-1}$ for all $\sigma \geq 1$ where
$\tilde \Theta$ depends only on $\Theta$ and $(M, g)$. 
 
\item By Corollary \ref{cor:surface_comparison} with $\beta =
\frac{\gamma}{2}$, $$\H^{n-1}_{g_m} (\partial \tilde \Omega_m) \leq \H^{n-1}_g(\partial
\tilde \Omega) +
O(\H^{n-1}_g(\partial \tilde \Omega)^{\frac{2-\gamma}{2(n-1)}}) \leq \H^{n-1}_g(\partial
\Omega)
+ O (\H^{n-1}_g (\partial \Omega)^{\frac{2-\gamma}{2(n-1)}}).$$ 

\item \label{item4} By Lemma \ref{lem:volume_comparison} with $\alpha =
\frac{2+\gamma}{2}$, $\CL^n_{g_m} (\tilde \Omega_m) = \CL^n_g(\Omega) +
O(\CL^n_g(\Omega)^{\frac{4-\gamma}{2n}})$. 

\item \label{item5} By Lemma \ref{lem:volume_comparison} with $\alpha =
\frac{2+\gamma}{2}$ and choice of $r$, $\CL^n_{g_m} (B_r \setminus B_{r_h})
=  \CL^n_{g_m} (B_r \setminus B_1) + O(1) = \CL^n_g (B_r \setminus B_1) +
O(\CL^n_g(B_r \setminus B_1)^{\frac{4-\gamma}{2n}}) = \CL^n_g (\Omega) +
O(\CL^n_g(\Omega)^{\frac{4-\gamma}{2n}})$. 

\item By (\ref{item4}) and (\ref{item5}) and choice of $r$ we have that
$\CL^n_{g_m} (\tilde \Omega_m) = \CL^n_{g_m} (B_r \setminus B_{r_h})
+ O(r^{\frac{4-\gamma}{2}})$. Let $\tilde r$ be such that $\CL^n_{g_m} (\tilde
\Omega_m) = \CL^n_{g_m} (B_{\tilde r} \setminus B_{r_h})$. Then
$\tilde r = r + O (r^{- n + \frac{6-\gamma}{2}})$. 

\item The Schwarzschild region $\tilde \Omega_m$ is $(\frac{1
+ \tau}{2}, \frac{\eta}{2})$-off-center provided that $\CL^n_g(\Omega)$ is
sufficiently large. Hence $\H^{n-1}_{g_m} (S_{\tilde r} )+  c \eta m   \left(1 -  \frac{1}{\tau}  \right)^2 \tilde r \leq \H^{n-1}_{g_m}
(\partial \tilde \Omega_m)$ by (\ref{eqn: comparison in Schwarzschild}).

\item $\H_{g_m}^2 (S_r) = \H^{n-1}_{g_m} (S_{r} ) \leq
\H^{n-1}_{g_m} (S_{\tilde r} ) +  O(\CL^n_g(\Omega)^{\frac{2-\gamma}{2n}})$ where
the inequality follows by explicit computation from $\CL^n_{g_m} (B_r \setminus
B_{r_h}) = \CL^n_g (\Omega) +
O(\CL^n_g(\Omega)^{\frac{4-\gamma}{2 n}})$. 

\item $\H^{n-1}_g(S_r) \leq \H^{n-1}_{g_m} (S_r) + O(1)$. This is obvious.
\item $\H^{n-1}_g (S_r) \leq \H^{n-1}_g(\partial \Omega) -  \frac{c}{2} \eta m 
\left(1 -  \frac{1}{\tau}  \right)^2 r+ O(\CL^n_g(\Omega)^{\frac{2-\gamma}{2n}}) +
O(\H^{n-1}_g(\partial \Omega)^{\frac{2-\gamma}{2(n-1)}})$.
\end{enumerate}
The conclusion follows from this since $\H^{n-1}_g(\partial \Omega)
^{\frac{1}{n-1}} \CL^n_g(\Omega)^{-\frac{1}{n}} \leq \Theta$ and since
$\frac{2-\gamma}{2} <1$.
\end{proof}


\section{Large isoperimetric regions center} \label{sec:center}

The results in this section follow from the effective volume comparison result in Theorem \ref{thm:effective_volume_comparison} and the results in Appendix \ref{sec:regularity} essentially as in \cite{Eichmair-Metzger:2010}. 

\begin{theorem} [Cf. \protect{\cite[Theorem 5.1]{Eichmair-Metzger:2010}}]\label{thm:centering} Let $(M, g)$ be an initial data set that is $\C^0$-asymptotic to Schwarzschild of mass $m > 0$. There exists a constant $V_0 > 0$ so that if $\Omega$ is an isoperimetric region with $\CL^n_g(\Omega) = V \geq V_0$, then $\Omega$ is smooth and  $\partial \Omega$ is a connected smooth hypersurface that is close to the coordinate sphere $S_r$, where $r$ is such that $\CL_g^n(\Omega) = \CL_g^n(B_r) = V$. The scale invariant $\C^{2, \alpha}$ norms of functions that describe such $\partial \Omega$ as normal graphs above the corresponding coordinate spheres $S_r$ tend to zero as $V \to \infty$. 

\begin{proof} It follows exactly as in the proof of Theorem 5.1 in \cite{Eichmair-Metzger:2010} that the reduced boundary of $\Omega$ outside of $B_{\frac{r}{2}}$ is a smooth connected closed hypersurface with the properties asserted for the boundary of $\Omega$ in the statement of the theorem. Assume that $  B_{r/2} \cap  \text{supp}(\partial^*\Omega) \neq \emptyset$ and let $\rho_0 := \sup \{ \rho \in [1, r/2] :  S_\rho  \cap \text{supp}  (\partial^*\Omega)  \neq \emptyset \}$. The half-space theorem \cite[Corollary 37.6]{GMT} shows that $S_{\rho_0} \cap \text{supp}  (\partial^*\Omega) \neq \emptyset$ consists of regular points. If $S_{\rho_0}$ is mean convex, this contradicts the maximum principle. Since all sufficiently large coordinate spheres are mean convex, we conclude that $\CL^n_g(B_{\frac{r}{2}} \setminus \Omega)$ is bounded independently of $V$. If it were non-zero, we could consider the smooth region $\Omega \cup B_{\frac{r}{2}}$ and move its mean convex outer boundary inwards to adjust the (relatively small) increase in volume back to $V$. The resulting region has less boundary area than $\Omega$, a contradiction. 
\end{proof}
\end{theorem}

\begin{theorem} [Cf. \protect{\cite[Theorem 5.2]{Eichmair-Metzger:2010}}]\label{thm:minimizers_exist} Let $(M, g) $ be an initial data set that is $\C^0$-asymptotic to Schwarzschild of mass $m>0$. There exists $V_0 >0$ so that for every volume $V \geq V_0$ there exists a smooth isoperimetric region  $\Omega$ with $\CL^n_g(\Omega) = V$. 
\end{theorem}

\begin{remark} It follows from the argument in \cite[Lemma 5]{Bray:1998} that a closed isoperimetric surface in a Riemannian manifold with non-negative Ricci-curvature is either connected or totally geodesic. It is tempting to impose a curvature condition and transplant Bray's argument to minimizing sequences (as in Proposition \ref{prop:cut_and_paste_for_minimizing_sequence}) to prevent them from splitting up into a part that stays behind and a part that diverges to infinity. Such arguments are investigated for various kinds of asymptotic geometries in recent work of A. Mondino and S. Nardulli. Note that the Ricci-tensor of the Schwarzschild manifold has a negative eigenvalue. Moreover, every complete one-ended asymptotically flat manifold that has non-negative Ricci-curvature is flat. (This follows from the Bishop-Gromov comparison theorem.)
\end{remark}


\section{Uniqueness of large isoperimetric regions and the existence of an isoperimetric foliation}
\label{sec:uniqueness}
Let $\tau \in (0, \frac{1}{2})$ and $R, C>1$. We consider the
Banach space $\mathcal{B}_{R, \tau, C}$ of tuples $(u, g)$, where $u \in \C^{2, \alpha}(S_R(0))$ is such that 
\begin{equation*}
  \sup_{S_r(0)} R^{-1} |u| + |D u| + R | D^2 u |  + R^{1 + \alpha} [D^2 u]_\alpha \leq \tau,
\end{equation*}
where the derivatives and norms are those of $S_r(0)$, and where $g_{ij}$ is a $\C^2$ metric on 
on $\overline B_{2R}(0) \setminus B_{
\frac{R}{2}}(0)$ such that for some $\gamma\in(0,1]$,
\begin{eqnarray}
  \label{eq:class_metric_assumption}
    |(g - g_m)_{ij}| + R |\partial_k (g - g_m)_{ij}| + R^2
    |\partial^2_{kl}(g -
    g_m)_{ij}|
    \leq C R^{2-n-\gamma}
    \\ \text { on }
    \overline B_{2R}(0) \setminus B_{\tfrac{R}{2}}(0) \text{ for all } i, j, k, l \in \{1, \ldots, n\}. \nonumber
\end{eqnarray}
Here, $(g_m)_{ij} = (1 + \frac{m}{2r^{n-2}})^{\frac{4}{n-2}} \delta_{ij}$ are the coefficients of the Schwarzschild metric of mass $m>0$. Given $(u, g) \in \B_{R, \tau, C}$, we will consider the surface $\graph (u) := \{ (1 + R^{-1} u(x)) x : x \in S_R(0)\} \subset \overline B_{2R}(0) \setminus B_\frac{R}{2}(0)$ and compute associated geometric quantities with respect to $g$. 

The classes $\B_{R, \tau, C}$ and how we use them are closely related to the classes $\B_{\sigma}$ in the work of G. Huisken and S.-T. Yau \cite[p. 286]{Huisken-Yau:1996}, cf. the proof of Theorem 5.1 and the remarks in the last paragraph on p. $311$ in their paper. 

\subsection{Curvature estimates for surfaces in $\B_{R, \tau, C}$}

\begin{proposition} \label{prop:geometry_sigma}
  Given $C>0$, there exist $R_0 >0$ and $\tau_0 \in (0, \frac{1}{2})$ such
that for all $R > R_0$ and $\tau \in (0, \tau_0)$, we have the following
estimates for geometric quantities of $\Sigma = \graph(u)$ with
respect to $g$ for all $(u, g) \in \B_{R, \tau, C}$:
  \begin{align}
     \nonumber
    |\tf| &\leq c (\tau R^{-1} + R^{1-n-\gamma}),
    \\
    \nonumber
    \frac{(n-1)}{2} \leq |H R| & \leq 2(n-1), 
    \\
       \nonumber
   |\Rm| &\leq c R^{-n}, 
   \\ \label{eq:curvest_RicciNN}
    \left|\Ric(\nu,\nu) + \frac{(n-1)(n-2)m}{R^n}\right| &\leq c(\tau R^{-n} + R^{-n-\gamma}), \text{and}
    \\
   \big|\iota^*_\Sigma (\Rm \lfloor \nu)\big| &\leq c (\tau R^{-n} + R^{-n-\gamma}).    \nonumber 
  \end{align}  
Here, $\Rm$ and $\Ric$ denote the Riemann and the Ricci curvature tensors of $g$, $\nu$ the unit normal of $\Sigma$ with respect to $g$, $\tf$ the trace free part of the second fundamental form of $\Sigma$, $H$ the mean curvature, $\iota_\Sigma$ the embedding of $\Sigma$ into $\R^n$, $\overline \nabla$ the covariant derivative with respect to the induced metric on $\Sigma$, and $c>0$ is a constant that only depends on $n$ and $C$. Contractions are taken with respect to the first index. Our sign conventions are reviewed in Appendix \ref{sec:standard}. 
\end{proposition}

\begin{lemma}[J. Simons' identity]  
  \label{lemma:simons_orig}
  Let $\Sigma$ be a hypersurface of a Riemannian manifold $(M, g)$ with induced metric
  $\gt_{ij}$, second fundamental form $h_{ij}$, and mean curvature $H = \bar g^{ij} h_{ij}$.  Then
  \begin{equation*}
    \begin{split}
      \Lapt h_{ij}
      &=
      \nabt^2_{ij} H
      + H h_i^k h_{kj}
      - |h|^2 h_{ij}
      + h_i^k {(\iota_\Sigma^* \Rm)^l}_{jlk}
      + h^{kl} (\iota_\Sigma^* \Rm)_{kijl}
      \\
      &\qquad
      + \nabt_j (\iota_\Sigma^* (\Ric \lfloor \nu)_i)
      + \nabt^k ( \iota_\Sigma^* (\Rm \lfloor \nu)_{ijk} ).
    \end{split}
    \end{equation*}
\end{lemma}

The corollary below follows from separating $\ff$ into its trace free and pure trace part, $\ff_{ij} = \tf_{ij} + \nth H \gt_{ij}$.

\begin{corollary}
  Assumptions as in Lemma \ref{lemma:simons_orig}. Then
  \begin{equation*}
    \begin{split}
      \Lapt \tf_{ij} = \ &
      (\nabt^2 H - \frac{\Lapt H}{n-1} \bar g)_{ij} + H(\tf_i^k \tf_{kj} -  \nth|\tf|^2 \gt_{ij}) +
      \nth H^2 \tf_{ij} - |\tf|^2 \tf_{ij}
      \\
      &
      + \tf^k_i (\iota_\Sigma^* \Rm)_{ljlk} + \tf^{kl} (\iota_\Sigma^* \Rm)_{kijl} \\
      &+ \nabt_j (\iota_\Sigma^* (\Ric \lfloor \nu)_i)
      + \nabt^k ( \iota_\Sigma^* (\Rm \lfloor \nu)_{ijk} )
    \end{split}
  \end{equation*}
 and
  \begin{equation}
    \label{eq:simons_magnitude}
    \begin{split}
      \frac{1}{2} \Lapt |\tf|^2
      = \ &
      \la \tf, \nabt^2 H \ra + |\nabt\tf|^2 + H\tr \tf^3  + \nth
      H^2 |\tf|^2 - |\tf|^4
      \\
      &
      + \tf^{ij} \tf_{ik} (\iota_\Sigma^* \Rm)_{ljlk} + \tf^{ij} \tf^{kl} (\iota_\Sigma^* \Rm)_{kijl} \\
      & + \tf^{ij} \nabt_j (\iota_\Sigma^* (\Ric \lfloor \nu)_i)+ \tf^{ij}
 \nabt^k ( \iota_\Sigma^* (\Rm \lfloor \nu)_{ijk} ).
    \end{split}
  \end{equation}
  Here, $\langle \cdot, \cdot \rangle$ denotes the inner product with respect to $\bar g$. 
\end{corollary}

\begin{lemma}
  \label{lem:estimate_calculation}
  Given $C>0$, there exist $R_0>0$ and $\tau_0\in(0,\half)$ such that
  for all $R>R_0$ and $\tau\in(0,\tau_0)$ the following holds: If $(u, g) \in \B_{R, \tau, C}$ is such that $\Sigma = \graph (u)$ has constant mean curvature $H$ with respect to $g$, and if $v$ is a non-negative Lipschitz function on $\Sigma$, then
  \begin{equation*}
    \begin{split}
      &\int_\Sigma \half \la \nabt v, \nabt |\tf|^2 \ra +
      \tfrac{1}{2(n-1)} v H^2 |\tf|^2 \dmu_g
      \\
      &\qquad
      \leq
      c \int_\Sigma
      |\nabt v| |\tf|| \iota_\Sigma^*(\Riem \lfloor \nu)|
      + v | \iota_\Sigma^*(\Riem \lfloor \nu) |^2 \dmu_g.
    \end{split}
  \end{equation*}
  Here, $c>0$ denotes a constant which only depends on $n$ and
  $C$. 
\begin{proof}  
  Using Proposition \ref{prop:geometry_sigma} we can estimate
  \begin{equation*}
    |H\tr \tf^3| + |\tf|^4 + |\tf^{ij}\tf_{ik} (\iota_\Sigma^* \Riem)_{ljlk}| +
    |\tf^{ij}\tf^{kl}(\iota_\Sigma^* \Riem)_{kijl}|
      \leq
      \tfrac{1}{2(n-1)} H^2 |\tf|^2 
  \end{equation*}
  provided that $\tau_0 \in (0, \frac{1}{2})$ is small enough and $R_0>0$ is large enough. In conjunction with
  \eqref{eq:simons_magnitude} we obtain the differential inequality
  \begin{equation*}
    - \half \Lapt|\tf|^2 + \tfrac{1}{2(n-1)} H^2 |\tf|^2  +
    |\nabt\tf|^2
    \leq
    - \tf^{ij} \nabt_j (\iota_\Sigma^* (\Ric \lfloor \nu)_i)- \tf^{ij}\nabt^k ( \iota_\Sigma^* (\Rm \lfloor \nu)_{ijk}             ).
  \end{equation*}
  We multiply this inequality with $v$ and integrate over $\Sigma$. Upon an integration by parts of the
  first term on the left and the two terms on the right, we obtain
  \begin{equation*}
    \begin{split}
      &\int_\Sigma \half \la \nabt v, \nabt |\tf|^2 \ra +
      \tfrac{1}{2(n-1)} v H^2 |\tf|^2 + v |\nabt\tf|^2 \dmu_g
      \\
      &\qquad
      \leq
      \int_\Sigma \nabt v * \tf *  \iota_\Sigma^*(\Riem \lfloor \nu) + v * \nabt \tf *
       \iota_\Sigma^*(\Riem \lfloor \nu) \dmu_g
      \\
      &\qquad
      \leq
      \int_\Sigma \nabt v * \tf * \iota_\Sigma^*(\Riem \lfloor \nu) + \half v |\nabt\tf|^2 + c v | \iota_\Sigma^*(\Riem \lfloor \nu)|^2 \dmu_g.
    \end{split}
  \end{equation*}
 \end{proof}
\end{lemma}

The result below corresponds to \cite[Lemma 5.6]{Huisken-Yau:1996},
where a variant of the iteration technique of
\cite{Schoen-Simon-Yau:1975} for volume preserving stable constant
mean curvature surfaces is applied to obtain curvature estimates. Here,
we use only that the surfaces have constant mean curvature, along with
J. Simons' identity and a standard Stampacchia iteration. Another ingredient
in our proof is an insight from \cite{Metzger:2007} related to an integration
by parts on certain covariant derivatives of curvature that appear
contracted with the tracefree part of the second fundamental form in
J. Simons' identity. This is applied to the effect that we get by
assuming that $(M, g)$ is $\C^2$-asymptotic to Schwarzschild of mass $m>0$, rather
than $\C^4$-asymptotic as in \cite{Huisken-Yau:1996}.

\begin{proposition}
  \label{prop:curvature_estimate}
  Given $C>0$, there exist $R_0>0$ and $\tau_0\in(0,\half)$ such that
  for all $R>R_0$ and $\tau\in(0,\tau_0)$ the following holds: If $(u, g) \in \B_{R, \tau, C}$ is such that $\Sigma = \graph (u)$ has constant mean curvature $H$ with respect to $g$,
  then
  \begin{equation*}
    |\tf| \leq c (\tau R^{1-n} + R^{1-n-\gamma}),
  \end{equation*}
  where $c$ is a constant that only depends on $C$ and $n$.
\end{proposition}

\begin{proof}
  We let $u := |\tf|^2$ and $s:=\sup_\Sigma u$. Let $u_k := (u-k)_+$,
  where $k\geq 0$ is a constant. Then $\Omega_k := \supp u_k$
  satisfies $\Omega_k \subset \{ u \geq k\}$. On $\Omega_k$ we have
  that $\nabt u_k = \nabt u$, on its complement we have $\nabt u_k = 0$ almost everywhere. Let $A(k):= \area_g(\Omega_k)$.
  
  Using $v= u_k$ as a test function in Lemma \ref{lem:estimate_calculation}, we obtain that
  \begin{multline*}
    \int_{\Omega_k} \half |\nabt u|^2 + \frac{1}{2(n-1)} H^2 |\tf|^2 u_k
    \dmu_g
    \\
    \leq
    c
    \int_{\Omega_k} |\nabt u| |\tf| | \iota_\Sigma^* (\Riem\lfloor \nu) | + u_k
    | \iota_\Sigma^* (\Riem\lfloor \nu)|^2 \dmu_g.
  \end{multline*}
  Using that $u_k \leq u \leq s$ we see that the right hand side can be estimated by 
  \begin{equation*}
    \begin{split}
      \frac{1}{4} \int_{\Omega_k} |\nabt u|^2 \dmu_g
      +
      c A(k) s (\tau^2 R^{-2n} + R^{-2n-2\gamma})
    \end{split}
  \end{equation*}
  so that 
  \begin{equation} \label{eq:estimate_3}
    \int_{\Omega_k} |\nabt u|^2 + H^2 u_k^2 \dmu_g
    \leq
    c A(k) s (\tau^2 R^{-2n} + R^{-2n-2\gamma}).
  \end{equation}
  In dimensions $n>3$ we continue the estimate as follows. We combine
  the H\"older inequality and the Michael-Simon-Sobolev inequality
  \cite[Theorem 2.1]{Michael-Simon:1973} to find that
  \begin{equation*}
    \begin{split}
      \int_{\Omega_k} u_k^2 \dmu_g
      &\leq
      A(k)^{2/(n-1)} \left(\int_{\Omega_k} u_k^{2(n-1)/(n-3)}
        \dmu_g\right)^{(n-3)/(n-1)}
      \\
      &\leq
      c A(k)^{2/(n-1)} \int_{\Omega_k}  |\nabt u|^2 + H^2 u_k^2 \dmu_g.
    \end{split}
  \end{equation*}
  In dimension $n=3$ we estimate
  \begin{equation*}
    \begin{split}
      \int_{\Omega_k} u_k^2 d \H^2_g
      &\leq
      \left(\int_{\Omega_k} |\nabt u| + H |u_k| d \H^2_g \right)^2
      \\
      &\leq
      c A(k) \int_{\Omega_k} |\nabt u|^2 + H^2 |u_k|^2 d \H^2_g.
    \end{split}
  \end{equation*}
  In conjunction with \eqref{eq:estimate_3} we obtain, in both cases,
  that is for all $n\geq 3$, that
  \begin{equation*}
    \int_{\Omega_k} u_k^2 \dmu_g
    \leq
    c A(k)^{(n+1)/(n-1)} s (\tau^2 R^{-2n} + R^{-2n-2\gamma}).
  \end{equation*}
  Let $h>k$. Then
  \begin{equation*}
    \begin{split}
      (h-k)^2 A(h)
      &\leq
      \int_{\Omega_h} u_k^2 \dmu_g
      \leq
      \int_{\Omega_k} u_k^2 \dmu_g
      \\
      &\leq
      c A(k)^{(n+1)/(n-1)} s (\tau^2 R^{-2n}+ R^{-2n-2\gamma}).
    \end{split}
  \end{equation*}
  Thus the function $A$ satisfies the iteration inequality from
  \cite[Lemma B.1]{Kinderlehrer-Stampacchia:2000} with $\alpha=2$,
  $\beta= \frac{n+1}{n-1}$, and $C=c s (\tau^2 R^{-2n}+
  R^{-2n-2\gamma})$. This lemma yields that $A(d) =0$ with
  \begin{equation*}
    d^2
    = 
    c s (\tau^2 R^{-2n} + R^{-2n-2\gamma}) A(0)^{2/(n-1)}.
  \end{equation*}
  Note that $A(0)^{2/(n-1)} = \area_g(\Sigma)^{2/(n-1)} \leq c
  R^2$. Hence
  \begin{equation*}
     \sup |\tf|^2 = s \leq d \leq c s^{1/2} (\tau R^{1-n} + R^{1-n-\gamma}). 
  \end{equation*}
 \end{proof}

\subsection{Eigenvalue estimates}

In this subsection we use the curvature estimate in Proposition \ref{prop:curvature_estimate} to derive precise estimates for the spectrum of the Jacobi operator on constant mean curvature spheres in $\B_{R,\tau,C}$. The strategy is that of \cite[Section 4]{Huisken-Yau:1996} adapted to arbitrary dimensions. Because the details of this adaption are delicate we present the full argument here.  

\begin{lemma}[\protect{Lichnerowicz, e.g. \cite[Section III.4]{Lectures-DG}}]
  \label{lemma:lichnerowicz}
  Assume that $\Sigma$ is a closed $(n-1)$-dimensional manifold with Riemannian metric
  $\gt$ and Ricci curvature $\Rict \geq \kappa \gt$, where $\kappa\geq0$. Then the first non-zero
  eigenvalue $\mu_0$ of $-\Lapt$ satisfies
  \begin{equation*}
    \mu_0 \geq \frac{n-1}{n-2} \kappa.
  \end{equation*}
\end{lemma}

\begin{proposition} [\protect{Cf. \cite[Lemma 3.13]{Huisken-Yau:1996}}]
  \label{prop:eigenvalue_lap} Given $C>0$, there exist $R_0 >0$ and
  $\tau_0 \in (0, \frac{1}{2})$ such that for all $R > R_0$ and $\tau
  \in (0, \tau_0)$, if $(u, g) \in \B_{R, \tau, C}$ and if $\Sigma =
  \graph(u)$ has constant mean curvature $H$ with respect to $g$, then
  the first eigenvalue $\mu_0$ of $-\Lapt$ on $\Sigma$ satisfies the
  estimate
  \begin{equation*}
    \mu_0 \geq \frac{H^2}{n-1}  + \frac{2(n-1) m}{R^n} - O(\tau R^{-n} +
R^{-n-\gamma}).
  \end{equation*}
\end{proposition}
\begin{proof}
  The Ricci curvature of the induced metric $\bar g$ on $\Sigma$ is
  \begin{equation}
    \label{eq:eigen1}
    \begin{split}
      \Rict_{ij}
      &=
      \Ric_{ij} - \Rm_{\nu ij\nu} + H \ff_{jk} - \ff_{ik} \ff_{jl}
      \\
      &=
      \Ric_{ij} - \Rm_{\nu ij\nu} + \frac{n-2}{(n-1)^2} H^2 \gt_{ij}
      + \tfrac{n-3}{n-1} H \tf_{ij} - \tf_{ik}\tf_{jk}. 
    \end{split}
  \end{equation}
  Note that since $g_m$ is scalar flat and conformally flat, the scalar and Weyl curvature of $g$ are of order
  $O(R^{-n-\gamma})$. In particular,
  \begin{equation*}
    \Riem_{\nu ij \nu} = \frac{1}{n-1} \Ric(\nu,\nu) \gt_{ij} +  O(R^{-n-\gamma}). 
  \end{equation*}
  Using Proposition \ref{prop:geometry_sigma}, we find that
  \begin{equation*}
    \begin{split}
      \Ric_{ij} - \Rm_{\nu ij\nu}
      &=
      -\frac{2}{n-1} \Ric(\del_r,\del_r) \gt_{ij} + O(\tau R^{-n} +
R^{-n-\gamma})
      \\
      &= 
      \frac{2(n-2) m}{R^n} \gt_{ij} + O(\tau R^{-n} + R^{-n-\gamma}).
    \end{split}
  \end{equation*}
  In view of Proposition \ref{prop:curvature_estimate}, the last two terms in \eqref{eq:eigen1} are
  of order $O(\tau R^{-n} + R^{-n-\gamma})$. The proposition follows from Lemma \ref{lemma:lichnerowicz} with $\kappa \geq \frac{2(n-2)
    m}{R^n} + \frac{n-2}{(n-1)^2} H^2 - O(\tau R^{-n} + R^{-n-\gamma})$. 
    \end{proof}

\begin{corollary} [Cf. \protect{\cite[Theorem 4.1]{Huisken-Yau:1996}}]
  \label{cor:eigenvalue}
  Assumptions as in Proposition \ref{prop:eigenvalue_lap}. Then the lowest
  eigenvalue of the Jacobi operator
  \begin{equation*}
    Lu = -\Lapt u - \big(|\ff|^2 + \Ric(\nu,\nu)\big)u
  \end{equation*}
  on functions with zero mean, that is 
  \begin{equation*}
    \lambda_1 := \inf\left\{ \int_\Sigma fLf \dmu_g \mid \int_\Sigma
        f^2\dmu_g = 1\ \text{and}\ \int_\Sigma f \dmu_g = 0  \right\},
  \end{equation*}
  satisfies
  \begin{equation*}
    \lambda_1 \geq \frac{n(n-1)m}{R^n} - O(\tau R^{-n} + R^{-n-\gamma}). 
  \end{equation*}
\begin{proof}
  Given $f$ with$\int_\Sigma f^2 \dmu_g =1$ and $\int_\Sigma f \dmu_g = 0$, we use Propositions  \ref{prop:eigenvalue_lap} and \ref{prop:curvature_estimate} and (\ref{eq:curvest_RicciNN}) to obtain that
  \begin{equation*}
    \begin{split}
      &\int_\Sigma fLf \dmu_g
      \\
      &=
      \int_\Sigma |\nabt f|^2 - f^2 (|\ff|^2 + \Ric(\nu,\nu)) \dmu_g
      \\
      &\geq
      \int_\Sigma f^2 \big( \tfrac{1}{n-1} H^2 + \tfrac{2(n-1) m}{R^n} 
      - |\ff|^2 - \Ric(\nu,\nu) - O(\tau R^{-n} + R^{-n-\gamma}) \big) \dmu_g
      \\
      &\geq
      \left( \frac{n (n-1) m}{R^n} - O(\tau R^{-n} + R^{-n-\gamma}) \right)
      \int_\Sigma f^2\dmu_g.
    \end{split}
  \end{equation*}
\end{proof}
\end{corollary}

\begin{proposition} [Cf. \protect{\cite[Theorem 4.1]{Huisken-Yau:1996}}]
  \label{prop:invertible}
  Assumptions as in 
  Proposition \ref{prop:eigenvalue_lap}. Then the Jacobi operator $L$ on
  $\Sigma$ is invertible with the explicit bound $\mu_1 \geq
  \frac{n(n-1)m}{R^n} - O(\tau R^{-n} + R^{-n-\gamma})$ for the eigenvalue of
least
  absolute value.
\end{proposition}
\begin{proof}
  Let $\mu_0$ be the smallest eigenvalue of $L$, that is
  \begin{equation*}
    \mu_0 := \inf \left\{ \int_\Sigma f L f \dmu_g \mid \int_\Sigma f^2
      \dmu_g = 1 \right\}.
  \end{equation*}
  Choosing $f$ to be the constant function $f=\area_g(\Sigma)^{-1/2}$, we
  find that
  \begin{equation}
    \label{eq:mu0_from_above}
    \mu_0 \leq -\frac{H^2}{n-1}  + \frac{(n-1)(n-2) m}{R^n} + O(\tau
R^{-n} + R^{-n-\gamma}),
  \end{equation}
  where we use the estimates from
  Proposition \ref{prop:curvature_estimate} and
  \eqref{eq:curvest_RicciNN}. On the other hand, by dropping
  the Dirichlet energy term and using \eqref{eq:curvest_RicciNN}, we find that
  \begin{equation*}
    \begin{split}
      \int_\Sigma fLf \dmu_g
      &\geq
      - \int_\Sigma (|\ff|^2 + \Ric(\nu,\nu)) f^2 \dmu_g
      \\
      &\geq
      \left( -\frac{H^2}{n-1} +  \frac{(n-1)(n-2) m}{R^n} - O(\tau R^{-n}
+ R^{-n-\gamma}) \right) \int_\Sigma f^2 \dmu_g.
    \end{split}
  \end{equation*}
 In conjunction with \eqref{eq:mu0_from_above}, this gives   \begin{equation*}
    \left| \mu_0 + \frac{H^2}{n-1} -  \frac{(n-1)(n-2) m}{R^n}
    \right|
    \leq O(\tau R^{-n} + R^{-n-\gamma}).
  \end{equation*}
  Let $h_0$ be the corresponding eigenfunction, so that $L h_0 = \mu_0
  h_0$. Denote by $\bar h_0$ its mean value. Multiply the eigenvalue
  equation by $(h_0-\bar h_0)$  and integrate to obtain
  \begin{equation}
    \label{eq:invertible_1}
    \begin{split}
      &\int_\Sigma (h_0-\bar h_0) L (h_0 - \bar h_0) \dmu_g - \mu_0
      \int_\Sigma (h_0 - \bar h_0)^2 \dmu_g
      \\
      &\qquad
      =
      \bar h_0 \int_\Sigma (h_0 - \bar h_0) (|\ff|^2 + \Ric(\nu,\nu))
      \dmu_g.
    \end{split}
  \end{equation}
  By Corollary \ref{cor:eigenvalue}, the first term on the left is
  non-negative, so we drop it. The second term on the left has the
  factor $-\mu_0$ which we can estimate from below $-\mu_0 \geq
  \frac{1}{2(n-1)} H^2$. On the right we use that the function $(h_0 -
  \bar h_0)$ is $L^2$-orthogonal to constant functions in combination
  with the fact that $H^2$ is constant and \eqref{eq:curvest_RicciNN}
  to infer that
  \begin{equation*}
    \int_\Sigma (h_0 - \bar h_0) (|\ff|^2 + \Ric(\nu,\nu)) \dmu_g
    =
     \int_\Sigma (h_0 - \bar h_0) (|\tf|^2 + O(\tau R^{-n} + R^{-n-\gamma}))
\dmu_g.
  \end{equation*}
  In view of \eqref{eq:invertible_1}, we obtain that
  \begin{equation*}
    \begin{split}
      &\frac{H^2}{2(n-1)}  \int_\Sigma (h_0-\bar h_0)^2 \dmu_g
      \\
      &\quad
      \leq
      \frac{H^2}{4(n-1)} \int_\Sigma (h_0-\bar h_0)^2 \dmu_g
      \\
      &\qquad
      + c |\bar h_0|^2 H^{-2} \int_\Sigma  \big(|\tf|^4 + O(\tau^2 R^{-2n}
+ R^{-2n-2\gamma})\big)\dmu_g.
    \end{split}
  \end{equation*}
  This implies the estimate
  \begin{equation}
    \label{eq:eigenfct_l2_estimate}
    \int_\Sigma (h_0-\bar h_0)^2 \dmu_g
    \leq
    O(\tau^2 R^{3-n} + R^{3-n-2\gamma}) |\bar h_0|^2.
  \end{equation}
  Let $\mu_1$ be the next eigenvalue of $L$ with corresponding
  eigenfunction $h_1$. We show that its mean value $\bar h_1$ is
  small. To this end, observe that $h_1$ is $L^2$-orthogonal to $h_0$
  and therefore
  \begin{equation*}
    0
    =
    \int_\Sigma h_0 h_1 \dmu_g
    =
    \int_\Sigma (h_1-\bar h_1) (h_0 - \bar h_0) \dmu_g + \bar h_0
    \int_\Sigma h_1\dmu_g.    
  \end{equation*}
  Hence we get from \eqref{eq:eigenfct_l2_estimate} that
  \begin{equation*}
    \left| \int_\Sigma h_1 \dmu_g \right|
    \leq
    O(\tau R^{\frac{3-n}{2}} + R^{\frac{3-n}{2}-\gamma}) \left( \int_\Sigma |h_1-\bar h_1|^2\dmu_g\right)^{1/2},
  \end{equation*}
  or equivalently
  \begin{equation}
    \label{eq:invertible_h1}
    |\bar h_1| \leq O(\tau R^{\frac{5-3n}{2}} +
    R^{\frac{5-3n}{2}-\gamma}) \left(\int_\Sigma |h_1-\bar h_1|\dmu_g\right)^{1/2}.
  \end{equation}
  Multiply the equation $L h_1 = \mu_1 h_1$ by $(h_1 - \bar h_1)$ and
  integrate. This yields
  \begin{equation*}
    \int_\Sigma (h_1 - \bar h_1 ) L (h_1 - \bar h_1) \dmu_g
    =
    \mu_1 \int_\Sigma (h_1 - \bar h_1)^2
    +
    \bar h_1 \int_\Sigma (h_1 - \bar h_1) (|\ff|^2 + \Ric(\nu,\nu)) \dmu_g.
  \end{equation*}
  The term on the left can be estimated below using
  Corollary \ref{cor:eigenvalue} by $(\frac{n(n-1) m}{R^n} - O(\tau R^{-n}
+ R^{-n-\gamma})) \int_\Sigma (h_1 - \bar h_1)^2 \dmu_g$. Exploiting as before
  that $(h_1 - \bar h_1)$ is $L^2$-orthogonal to constant functions to
  estimate the second term on the right, we arrive at
  \begin{equation*}
    \begin{split}
      & \left(\frac{n(n-1) m}{R^n} - O(\tau R^{-n} + R^{-n-\gamma}) \right)
\int_\Sigma (h_1
      - \bar h_1)^2 \dmu_g
      \\
      &\qquad \leq
      \mu_1 \int_\Sigma (h_1 - \bar h_1)^2 \dmu_g
      +
      O(\tau R^{-n} + R^{-n-\gamma}) |\bar h_1| \int_\Sigma |h_1 - \bar
h_1|\dmu_g
      \\
      &\qquad \leq
      \mu_1 \int_\Sigma (h_1 - \bar h_1)^2 \dmu_g
      +
      O(\tau^2 R^{2-2n} + R^{2-2n-2\gamma})\int_\Sigma (h_1 - \bar h_1)^2 \dmu_g.
    \end{split}
  \end{equation*}
  Here, the second inequality follows from combining
  \eqref{eq:invertible_h1} with Cauchy-Schwarz to estimate the
  $L^1$-norm in terms of the $L^2$-norm.  In conclusion, we
  arrive at the estimate
  \begin{equation*}
    \mu_1 \geq \frac{n(n-1) m}{R^n} - O(\tau R^{-n} + R^{-n-\gamma}).
  \end{equation*}
  Since $\mu_1$ is positive, all other eigenvalues of $L$ are also
  positive. Moreover, since $|\mu_0| > |\mu_1|$, the eigenvalue $\mu_1$
  is the one with the least absolute value.
\end{proof}


\subsection{The uniqueness argument}
Here we adapt an idea from \cite[Section 6]{Metzger:2007} to derive uniqueness of constant mean curvature surfaces of a given mean curvature in the class $\B_{R, \tau,C}$. The advantage of this approach over the method in \cite[Section 4]{Huisken-Yau:1996} is that we need only require the metric to be $\C^2$-asymptotic to Schwarzschild of mass $m>0$, instead of $\C^4$-asymptotic. Also, volume preserving stability is not a hypothesis in our approach, but part of the conclusion, cf. Corollary \ref{cor:eigenvalue}.
 
\begin{proposition}
  \label{prop:deformation}
  Given $C>0$ there exists $\tau_0>0$ such that for every $\tau \in (0,
  \tau_0)$ there exists $R_0>0$ so that for all $R > R_0$ the following
  holds: Let $(u,g) \in B_{R,\tau/2, C}$ be such that $\Sigma = \graph u$ has constant mean curvature
  $H$ with respect to $g$. There exists a differentiable $1$-parameter
  family of functions $u_t \in \C^{2,\alpha}(S_R(0))$ such that $u_0 =
  u$, such that each $\Sigma_t = \graph(u_t)$ for $t\in[0,1]$ has
  constant mean curvature $H$ with respect to the metric $g_t := t g_m
  + (1-t) g$, and such that $(u_t,g_t) \in B_{R, \tau, C}$.
\end{proposition}
\begin{proof}
  Clearly, $g_t$ satisfies \eqref{eq:class_metric_assumption} with the
  same constant $C$ as $g$. We choose $\tau_0>0$ and $R_0 >0$ as in
  Propositions \ref{prop:curvature_estimate} and
  \ref{prop:invertible}. In particular, the Jacobi operator $L$ on
  $\graph(u)$ with respect to $g$ is invertible for all $(u,g) \in
  \B_{R, \tau ,C}$, provided that $R > R_0$ and $\tau \in (0, \tau_0)$.

  Consider the set $I \subset (0,1]$ of all $t\in (0,1]$ such that there exists a differentiable curve
  \begin{equation*}
    [0,t) \to \C^{2,\alpha}(S_R(0)) : s\mapsto u_s 
  \end{equation*}
  with $0 \mapsto u$, and such that $(u_s, g_s)\in \B_{R, \tau, C}$ and $\graph u_s$ has
  constant mean curvature $H$ with respect to $g_s$ for all $s \in [0, t)$.

  Consider the operator
  \begin{equation*}
    \CH : \C^{2,\alpha}(S_R(0)) \times [0,t] \to \C^{0,\alpha}(S_R(0))
  \end{equation*}
  which maps a pair $(v,s)$ to the mean curvature of $\graph v$ with
  respect to the metric $g_s$. The derivative of $\CH$ with respect to
  the first variable is given by the Jacobi operator $L_{(v, s)}$ of
  $\graph v$ with respect to the metric $g_s$. By assumption, $\graph(u)$ has constant mean curvature $H$ with respect to $g$. By
  Proposition \ref{prop:invertible}, the operator $L_{(u,0)}$ is
  invertible. It follows that $I$ contains a neighborhood of $0$.

  Let $t\in I$ and $s\mapsto u_s$ be the corresponding curve for $s\in
  [0,t)$. Standard compactness theory for solutions of the parametric constant mean curvature equation shows that there exists
  a limit $u_t\in\B_{R,\tau,C}$ of $u_s$ as $s \nearrow t$ such that $\graph(u_t)$
  has constant mean curvature $H$ with respect to $g_t$. Applying Proposition \ref{prop:invertible} and the implicit function theorem as in the preceding paragraph, we see that the curve $s\to
  u_s$ can be continued differentiably beyond $t$. It remains to show that this extended part remains in $\B_{R,\tau,C}$ so long as $t \leq 1$.

  We show that our assumptions
  imply that in fact $(u_t,g_t)\in \B_{R, 3\tau/4, C}$, so that by
  continuity the extension above remains in $\B_{R, \tau,
    C}$. Differentiating the equation $\CH(u_s,g_s) = H$ with respect to $s$ for
  $s\in[0,t]$ we get
  \begin{equation*}
    L_{(u_s,s)} \left(\frac{\del u_s}{\del s} \la \del_r, \nu_s \ra\right)
    =
    - D_2 \CH(u_s, g_s),
  \end{equation*}
  where $D_2 \CH (u_s,g_s)$ denotes the variation of $\CH$ with
  respect to $s$ and $\nu_s$ denotes the normal to $\graph(u_s)$ with
  respect to $g_s$. We calculate that $D_2 \CH(u_s, g_s) =
  O(R^{1-n-\gamma})$. From Proposition \ref{prop:invertible} we know that
   the
  norm of the inverse of $L_{(u_s,s)}$ is bounded by
  $\frac{2R^nm}{n(n-1)}$. It follows that 
  $\frac{\del u_s}{\del s} = O^{2, \alpha}(R^{1-\gamma})$ and hence that
  $$\sup_{S_R(0)} R^{-1}| u_t - u| + |D(u_t - u)| + R |D^2 (u_t-u)| + R^{1+\alpha} [D^2 (u_t - u)]_\alpha \leq C' R^{-\gamma}.$$
  Choosing $R_0$ even larger, if necessary, we can ensure that the right hand side is less than $\tau/4$, so that indeed $u_t \in \B_{R, \frac{3 \tau}{4}, C}$. 
\end{proof}

The proof of the following is now the same as that of \cite[Theorem 6.5]{Metzger:2007}:

\begin{theorem} \label{thm:uniqueness}
Given $C>0$, there exists $\tau \in (0, \frac{1}{2})$, $R_0 >0$, and $C'>0$ so that for all $R \geq R_0$ there exists exactly one $(u, g) \in \B_{R, \tau, C}$ such that $\Sigma = \graph(u)$ has the same constant mean curvature with respect to $g$ as $S_{R}(0)$ with respect to $g_m$. Moreover, we have that $(u, g) \in \B_{R, C' R^{-\gamma}, C}$, i.e. 
\begin{equation*}
\sup_{S_R(0)} R^{-1} |u| + |D u| + R | D^2 u |  + R^{1 + \alpha} [D^2 u]_\alpha \leq C' R^{- \gamma}.
\end{equation*}  
 \begin{proof}
  We choose the constants $\tau_0>0$ and $R_0>0$ as in
  Proposition \ref{prop:deformation}.  Assume that there are two such
  surfaces $\Sigma^i = \graph(u^i)$ for
  $i=1,2$. Proposition \ref{prop:deformation} implies that we can
  deform both surfaces $\Sigma^i$ along differentiable paths $u_t^i$ to surfaces of constant mean curvature $H$ in the
  Schwarzschild metric $g_m$. By the Alexandrov theorem in
  Schwarzschild proven by S. Brendle in \cite{Brendle:2011} and the
  fact that $R$ is large, we find that $u^1_1 = u^2_1 = 0$. Since the
  implicit function theorem gives local uniqueness, this implies
  that $u^1_t = u^2_t$ for all $t\in [0,1]$, in particular for $t=0$. That $(u, g) \in \B_{R, C'R^{- \gamma}, C}$ now follows from the estimates at the end of the proof of Proposition \ref{prop:deformation}. Reading the argument backwards gives the existence of $u$. 
\end{proof}
\end{theorem}

\subsection{Existence of an isoperimetric foliation} 
Theorem \ref{thm:uniqueness} and the invertibility of the Jacobi operator proven in Proposition \ref{prop:invertible} show that the asymptotic regime of an initial data set $(M, g)$ that is $\C^2$-asymptotic to Schwarzschild of mass $m>0$ is foliated by strictly volume preserving stable constant mean curvature surfaces. The existence of such a foliation was proven for dimension $n=3$ in \cite{Metzger:2007} and, earlier, in \cite{Huisken-Yau:1996, Ye:1996} under stronger asymptotic conditions. Theorem \ref{thm:centering} shows that in such an initial data set, large isoperimetric regions exist for every sufficiently large volume, and that their boundaries are constant mean curvature surfaces to which the uniqueness assertion in Theorem \ref{thm:uniqueness} applies. In summary, we obtain the following result:

\begin{theorem} \label{thm:existence}
  Let $(M, g)$ be an initial data set that is $\C^2$-asymptotic to Schwarzschild of mass $m>0$ at rate $\gamma \in (0, 1]$. There
  exist $V_0, C' >0$ with the following properties: For every $V \geq V_0$ there exists a unique isoperimetric region $\Omega_V \subset M$ of volume $V$. The boundary of $\Omega_V$ is a connected smooth embedded closed strictly volume preserving stable constant mean curvature surface $\Sigma_V$ of mean curvature $H_V$. We have that $U \subset \Omega_V$ and hence that $\Sigma_V \subset M \setminus U \cong_x \R^n \setminus B_{\frac{1}{2}} (0)$. Let $R>0$ large be chosen such that $S_{R}(0)$ has constant mean curvature $H_V$ with respect to $g_m$. There exists $u_V \in \C^{2, \alpha} (S_R(0))$ with 
  \begin{equation*}
\sup_{S_R(0)} R^{-1} |u| + |D u| + R | D^2 u |  + R^{1 + \alpha} [D^2 u]_\alpha \leq C' R^{- \gamma}
\end{equation*}  
  and such that $\Sigma_V = \graph(u_V)$. The isoperimetric surfaces $\{\Sigma_V\}_{V
    \geq V_0}$ form a smooth foliation of the region $M \setminus \Omega_{V_0}$. Here, $U \subset M$ and $C>0$ are as in Definition \ref{def:initial_data_sets}. 
\end{theorem}


\section{The center of mass}
\label{sec:center-mass}
In this section we let $(M,g)$ be $\C^2$-asymptotic to Schwarzschild of mass $m>0$, and we assume that $(M,g)$ is asymptotically even. Under these conditions\footnote{The condition that $\Scal(x) - \Scal(-x) = O(|x|^{-n-1-\gamma})$ ensures that this limit exists. Cf. with the proof of Lemma \ref{lemma:alternative_center}.} the limits \begin{equation*}
  \mathcal{C}_l = \frac{1}{2m (n-1) \omega_{n-1}} \lim_{r\to\infty}
  \frac{1}{r} \int_{S_r} \sum_{i,j=1}^n x_l \big(g_{ij,i} -
  g_{ii,j} \big)x_j - \sum_{i=1}^n (g_{il}x_i - g_{ii}x_l) \dmu_{\delta}
\end{equation*}
exist for every $l\in\{1,\dots,n\}$. $\mathcal{C} =
(\mathcal{C}_1, \ldots, \mathcal{C}_n)$ is called the center of mass of $(M, g)$. See \cite{Regge-Teitelboim:1974, Beig-OMurchadha:1987} for the origin of the concept of the center of mass for three dimensional initial data sets, and to the recent papers \cite{Corvino-Wu:2008, Huang:2009, Huang:2010} for the relationship of this definition with other, partially equivalent geometric notions.

The main result of this section is that under the given assumptions
the centers of mass of the surfaces $\Sigma_H$ with $H\in(0,H_0)$
converge to $\mathcal{C}$ as $H\searrow 0$. The convergence of the
centers of mass of the surfaces has been established in dimension
$n=3$ for metrics that are $\C^4$-asymptotic to Schwarzschild with
mass $m>0$ by G. Huisken and S.-T. Yau \cite[Theorem
4.2]{Huisken-Yau:1996}. L.-H. Huang showed in dimension $n=3$ that
this limit equals the center of mass of $(M,g)$ if the latter exists
\cite[Theorem 2]{Huang:2009}. In the following we generalize this
result to arbitrary dimension. Even when $n=3$, our proof here is different from the one in \cite{Huang:2009}, cf. Remark \ref{rem:Huangproof}. 
\begin{theorem}
  \label{thm:center_of_mass}
  Let $(M,g)$ be an initial data set that is $\C^2$-asymptotic to
  Schwarzschild of mass $m>0$, and assume that $(M, g)$ is
  asymptotically even. Let $\{\Sigma_V\}_{V \geq V_0}$ be the isoperimetric surfaces of 
  Theorem \ref{thm:existence}. For $l\in\{1,\dots,n\}$, define
  \begin{equation*}
    a(V)_l := \area_\delta(\Sigma_V)^{-1} \int_{\Sigma_V} x_l \dmu_\delta.
  \end{equation*}
  Then $a(V)_l \to \mathcal{C}_l$ as $V \to \infty$, where
  $\mathcal{C} = (\mathcal{C}_1, \ldots, \mathcal{C}_n)$ denotes the center of mass of $(M,g)$.
\end{theorem}

\begin{proof}
In this proof, $c$ is a constant depending only on $(M, g)$ that may vary from line to line. We may parametrize the surfaces $\{ \Sigma_V \}_{V\geq V_0}$ by their constant mean curvatures $H = H(V)$, or, equivalently, by $R = R(V)$, where $R$ is such that $S_R(0)$ has constant mean curvature $H$ with respect to $g_m$. From Theorem \ref{thm:existence} and Proposition \ref{prop:curvature_estimate} (with $\tau = C' R^{-\gamma}$) we know that $\lim_{V \to \infty} H \left(\frac{n V}{\omega_{n-1}}\right)^\frac{1}{n} = n-1 = \lim_{V \to \infty} H R$, and that $|\tf|_g \leq c R^{1-n-\gamma}$, where $\tf_g$ is the traceless part of the second fundamental form of $\Sigma_R$. It follows that $\sup_\Sigma |\tf_\delta|_\delta + \sup_\Sigma |H_\delta - H| \leq c R^{1-n}$. By Lemma \ref{lem:sphere_approximation}, there exists a coordinate sphere $S_H := S_{r_H}(p_H)$ and $v_H \in \C^2( S_{H})$ such that
  $\Sigma_H = \{ F_H (y): y \in S_H\}$, where $F_H : S_H \to \R^n$ is
  given by $F_H(y) := y + v_H (y) \frac{y - p_H}{r_H}$, and
  such that\footnote{If $(M, g)$ is $\C^3$-asymptotic to Schwarzschild of mass $m>0$, we can derive the estimate $|\nabla_g \tf|_g \leq c R^{-n-\gamma}$ using a maximum principle argument (the Bernstein trick, differentiating J. Simons' identity), and argue as in the proof of \cite[Proposition 2.1]{Huisken-Yau:1996} to improve the right hand side of this estimate to $R^{2-n-\gamma}$. This would improve the subsequent estimates and we could treat the cases $n=3$ and $n \geq 4$ in one step.}
  \begin{equation}
    \label{eq:cm_sphere_approx}
    \sup_{S_H} r_H^{-1} |v_H| + |D v_H| + r_H |D^2 v_H| \leq c R^{2-n}.  \end{equation}
  It follows that $|\frac{r_H}{R} - 1| \leq c R^{2-n}$, that $|p_H|/r_H \leq c R^{-\gamma}$, and that 
  \begin{equation*}\
    |a(H) - p_H|
    \leq
    c R^{3-n}.
  \end{equation*}
  By Lemma \ref{lemma:alternative_center},
  \begin{multline}
    \label{eq:lans_center_of_mass}
    \left|  m (n-1) \omega_{n-1} (p_H-\mathcal{C})_l
      -  \int_{S_H}  (x - p_H)_l (H^{S_H} - \tfrac{n-1}{r_H}) \dmu_\delta
      \right|
      \\
      \leq
      c ( R^{1-2\gamma} + R^{-\gamma}).
  \end{multline}
  We now analyze the integral on the left hand side to show that it
  tends to zero as $H \to 0$. The same argument as in \cite[Theorem
  5.1]{Huisken-Yau:1996} shows that for a given vector $b\in\R^n$, we
  have that
  \begin{equation}
    \label{eq:cm_initial integral}
    0 = \int_{\Sigma_H} (H^{\Sigma_H} - H^{\Sigma_H}_\delta)\, \delta( b, \nu^{\Sigma_H}_\delta)\dmu_\delta.
  \end{equation}
  Here, we use the subscript $\delta$ to indicate that
  a geometric quantity is computed with respect to the Euclidean
  metric. We claim that
  \begin{equation}
    \label{eq:cm_target_integral}
    \left| \int_{S_H} (H^{S_H} - H^{S_H}_\delta)\, \delta (b, \nu_\delta^{S_H}) \dmu_\delta\right|
    \leq c R^{2-n}.
  \end{equation}
  To see this, we use the diffeomorphism $F_H : S_H \to \Sigma_H$ to
  pull back the integral \eqref{eq:cm_initial integral} to
  $S_H$. Using \eqref{eq:cm_sphere_approx}, we obtain that
  \begin{equation} \label{eqn:cm_sphere_approx_consequence}
    \begin{split}
      | \delta (\nu^{S_H}_\delta (y), b) -  \delta
      (\nu^{\Sigma_H}_\delta (F_H(y)), b)| 
      &\leq c R^{2-n}, \\
      \left| F^*_H (\H^{n-1}_\delta \lfloor {\Sigma} ) -
        \H^{n-1}_\delta \lfloor {S_H}\right| 
      &\leq c R^{2-n}, \text{ and } \\
      \big| (H^{\Sigma_H} - H_\delta^{\Sigma_H}) - (H^{S_H} -
      H_\delta^{S_H}) \big|
      &\leq c R^{3-2n}. 
    \end{split}
  \end{equation}
  In conjunction with \eqref{eq:cm_initial integral}, this gives
  \eqref{eq:cm_target_integral}.
  
  Note that $H^{S_H}_\delta = \frac{n-1}{r_H}$ and $\nu^{S_H}_\delta
  (x)= r_H^{-1}(x-p)$ for $x \in S_H$. If we let $b = e_l$ be a
  coordinate vector, it follows from \eqref{eq:cm_target_integral} and
  \eqref{eq:lans_center_of_mass} that
  \begin{equation}
    \label{eq:cm_centerest1}
    \left|  (p_H-\mathcal{C})_l \right|
    \leq
    c ( R^{1-2\gamma} + R^{-\gamma} + R^{3-n} ).    
  \end{equation}
  The second error term tends to zero as $H\to \infty$. We will assume for the moment that $n\geq 4$, so
  that the third error term tends to zero as $H\to\infty$. If also $\gamma>\frac12$, then we are done. If this is not the case
  then \eqref{eq:cm_centerest1} still gives that for $H \in (0, H_0)$ sufficiently small we have that
  \begin{equation*}
    |p_H| \leq c R^{1-2\gamma}.
  \end{equation*}
  We are now in the position to apply Lemma \ref{lemma:alternative_center} with
  $\gamma_1=2\gamma$ to improve \eqref{eq:cm_centerest1} to
  \begin{equation*}
    \left|  (p_H-\mathcal{C})_l \right|
    \leq
    c ( R^{1-3\gamma} + R^{-\gamma} + R^{3-n} ).
  \end{equation*}
  Repeating this argument a finite number of times, we find that
  \begin{equation*}
    \left|  (p_H-\mathcal{C})_l \right|
    \leq
    c ( R^{-\gamma} + R^{3-n} ).
  \end{equation*}
  This concludes the proof in the case $n\geq 4$.

  We will use a different method to approximate $\Sigma_H$ by a sphere
  when $n=3$, following \cite[Proposition 4.3]{Metzger:2007}, to
  improve our estimates in this dimension. First, note that the
  curvature estimate $|\tf_g | \leq c R^{-2-\gamma}$ implies that
  $|\tf_{g_m}|_{g_m} \leq cR^{-2-\gamma}$. Using conformal invariance,
  we see that
  \begin{equation*}
    \int_{\Sigma_H} |\tf_\delta|_\delta^2 d \H^2_\delta \leq c R^{-2-2\gamma}. 
  \end{equation*}
  Results of C. De Lellis
  and S. M\"uller \cite{DeLellis-Muller:2005,DeLellis-Muller:2006}
  imply that there exists a sphere $S_H
  = S_{r_H}(p_H) \subset \R^3$ with $r_H =
  \sqrt{\mathcal{H}_\delta^2(\Sigma)/4\pi}$ and a conformal map
  $\psi_H : S_H \to \Sigma_H\subset\R^3$ with conformal factor $w_H$
  such that $\psi_H^* \delta = w_H^2 \delta|_{S_H}$. Moreover,  we have that \begin{eqnarray*}  
    \| H_\delta^{\Sigma_H} - \frac{2}{r_H} \|_{L^2(\Sigma_H)} &\leq& cR^{-1-\gamma},
    \\
    \sup_{S_H} | \psi_H - \id | &\leq& c R^{-\gamma},
    \\
    \| \nu_\delta^{\Sigma_H} \circ \psi - \nu_\delta^{S_H} \|_{L^2(S_H)} &\leq& cR^{-\gamma}, \text{ and}
    \\
    \| w_H - 1 \|_{L^2(S_H)} &\leq& cR^{-\gamma}.
  \end{eqnarray*}
  The first three of these estimates are immediate from \cite[Theorem
  1.1]{DeLellis-Muller:2005}, using Sobolev embedding and
  rescaling. The last estimate follows from \cite[Theorem
  1.1]{DeLellis-Muller:2006}. Using these estimates in place of (\ref{eqn:cm_sphere_approx_consequence}) in the above argument, we obtain that 
  \begin{equation*}
    \label{eq:cm_target_integral_3d}
    \left| \int_{S_H} (H^{S_H} - H^{S_H}_\delta)\, \delta (b, \nu_\delta^{S_H}) d \H^{2}_\delta\right|
    \leq c R^{-1-\gamma} \text{ and }
    |a(H) - p_H|
    \leq
    c R^{-\gamma}. 
  \end{equation*} 
The rest of the argument proceeds exactly as in the case $n \geq 4$.    
\end{proof}
  
\begin{remark} \label{rem:Huangproof} We give a brief outline of L.-H. Huang's proof of Theorem \ref{thm:center_of_mass} in dimension $n=3$. First, the foliation through volume preserving stable constant mean curvature spheres that is constructed in \cite{Huisken-Yau:1996} using volume preserving mean curvature flow coincides with the foliation found in \cite{Ye:1996}. R. Ye's construction of the leaves in \cite{Ye:1996} proceeds by perturbing large coordinate spheres in Euclidean space to constant mean curvature surfaces with respect to the asymptotically flat metric. The obstruction to accomplishing this stems from the translational symmetries of Euclidean space, which account for the co-kernel of the linearized mean curvature operator on volume preserving deformations. The integral in \eqref{eq:lans_center_of_mass} measures the part of the ``error" that lies in the co-kernel. This obstruction vanishes on Euclidean coordinate sphere whose center is (close) to the center of mass of the initial data set. The size of the required perturbations is then so small that the centers of gravity of the constant mean curvature surfaces approach the centers of the Euclidean spheres they are constructed from, as their diameter tends to infinity.   
\end{remark}


\section{The isoperimetric mass} \label{sec:isomass}

Throughout this section, we let $(M, g)$ be a three dimensional initial data set as in Definition \ref{def:initial_data_sets}, except that we allow $M$ to have a non-empty compact boundary. If $\partial M \neq \emptyset$, we assume that $\partial M$ is a minimal surface, and that there are no other compact minimal surfaces in $M$. We modify the definition of the isoperimetric profile function $A_g : [0, \infty) \to \R$ as follows: 
\begin{eqnarray*} A_g(V) := \inf \{ \H^2_g(\partial^*\Omega \cap M) :  \Omega \subset \hat M \text{ is a Borel set with finite perimeter} \\ \text{containing } \hat M \setminus M, \text{ and } \CL^3_g(\Omega \cap M) = V\}. \end{eqnarray*}
Here, $\hat M$ is an extension of $M$ across its boundary, if $\partial M \neq \emptyset$, and $\hat M = M$ otherwise. Note that $A_g(V)$ is independent of the choice of $\hat M$. The classical regularity theory for this minimization problem with volume constraint and obstacle is discussed with precise references to the literature in \cite[Section 4] {Eichmair-Metzger:2010}. For $V>0$ the minimizers have a smooth boundary that is disjoint from $\partial M$. As before, minimizers will be called isoperimetric regions. 
  
\begin{definition} [G. Huisken \protect{\cite{Huisken-pc, Huisken:2006-iso}}] \label{def:isomasshuisken}
  We say that a sequence of smooth bounded regions
  $\{\Omega_i\}_{i=1}^\infty$ is an exhaustion of $(M, g)$, if $\Omega_i \subset \Omega_{i+1}$ for each
  $i=1, 2, \ldots$ and if $\bigcup_{i=1}^\infty \Omega_i = M$. The isoperimetric mass of $(M, g)$ is defined as
  \begin{eqnarray} \label{def:Huiskeniso}
    m_{\text{iso}} (M, g) := \sup \Big \{ \limsup_{i \to \infty} \frac{2}{\H^2_g(\partial \Omega_i)} \left( \CL^3_g (\Omega_i) - \tfrac{1}{6 \sqrt \pi} \H^2_g(\partial \Omega_i)^{\frac{3}{2}}\right) : \\ \{\Omega_i\}_{i=1}^\infty \text{ is an exhaustion of } (M, g) \Big \}. \nonumber
  \end{eqnarray}
\end{definition}

\begin{theorem}[G. Huisken  \protect{\cite{Huisken-pc}, see also \cite{Huisken:2006-iso}}] \label{thm:isomasshuisken}
  Let $(M,g)$ be a three dimensional initial data set. Assume that the scalar curvature of $(M, g)$ is non-negative. Then $m_\text{iso}(M,g)\geq 0$. Equality holds if, and only if, $(M,g)=(\R^3,\sum_{i=1}^3 dx_i^2)$.
\end{theorem}

Subsequent to the work of G. Huisken, it was observed by X.-Q. Fan, P. Miao, Y. Shi, and L.-T. Tam in \cite {Fan-Shi-Tam:2009} that the ``$\limsup$" in G. Huisken's definition (\ref {def:Huiskeniso}) recovers the ADM-mass of the initial data set when evaluated along exhaustions by concentric coordinate balls in an asymptotic coordinate system. In particular, $m_{iso}(M, g) \geq m_{ADM}(M, g)$ and the conclusion of Theorem \ref{thm:isomasshuisken} is seen to be a consequence of the positive mass theorem. 

\begin{definition} We define the modified isoperimetric mass of $(M, g)$ as 
  \label{def:isomass_alldim}
  \begin{equation*}
    \tilde m_{\text{iso}} (M,g)
    :=
    \limsup_{V\to\infty}
  \frac{2 }{A_g(V)}
      \left( V  - \frac{1}{6 \sqrt{\pi}} A_g(V)^{\frac{3}{2}}\right). 
  \end{equation*}
\end{definition}

It is easy to see that $\tilde m_{iso} (M, g) \geq m_{iso}(M, g)$. In particular, Theorem \ref{thm:isomasshuisken} holds with $m_{iso}(M, g)$ replaced by $\tilde m_{iso}(M, g)$. 

\begin{proof} [Proof of Theorem  \ref{thm:isomass_exist1}] In view of Theorem \ref{thm:isomasshuisken}, we may assume that $\tilde m_{iso} (M, g) >0$. Let $V>0$. By Proposition \ref{prop:cut_and_paste_for_minimizing_sequence}, there exists an isoperimetric region $\Omega \subset  M$ (it may be empty) and a sequence of coordinate balls $B(p_i, r_i)$ with $p_i \to \infty$ and $0 \leq r_i \to r \in [0, \infty)$ as $i \to \infty$ such that $\CL^3_g(\Omega) + \CL^3_g(B(p_i, r_i)) = V$ and such that $\H^{2}_g(\partial \Omega) + \H^{2}_g(\partial B(p_i, r_i)) \to A_g(V)$. Our goal is to show that by choosing $V$ sufficiently large we can arrange for $\CL^3_g(\Omega)$ to be greater than any given threshold. Let $V>0$ large be such that 
  \begin{equation*}
    \frac{2 }{A_g(V)}
      \left( V  - \frac{1}{6 \sqrt{\pi}} A_g(V)^{\frac{3}{2}}\right)
    > \frac{\tilde m_\text{iso} (M, g)}{2}.
  \end{equation*}
  Combining this with the lower bound $A_g(V)\geq 4 \pi r^2$
  we obtain the estimate
  \begin{equation*}
   \CL^3_g(\Omega) \geq  \frac{ \tilde m_\text{iso} (M, g) }{4}  A_g(V). 
  \end{equation*}
 That $A_g(V) \to \infty$ as $V \to \infty$ follows from Lemma \ref{lem:crudeisoperimetricinequality}. 
 \end{proof}
 

\appendix

\section{Integral decay estimates} \label{sec:integral_decay_estimates}

Our computations in this appendix take place in the part of an initial
data set $(M, g)$ that is diffeomorphic to $\R^n \setminus B_1(0)$
and where \begin{eqnarray*}
 r |g_{ij } - \delta_{ij}| \leq C \text{ for all } r
  \geq 1. \end{eqnarray*} For Corollary \ref{lem:volume_comparison} we
require in addition that for some $\gamma \in (0, 1]$, 
\begin{eqnarray} 
\label{eqn: decay assumptions integral decay Schwarzschild}
r^{n-2+\gamma} |g_{ij} - \left(1 + \frac{m}{2r^{n-2}} \right)^{\frac{4}{n-2}}
\delta_{ij}| \leq C \text{ for all } r \geq 1,
\end{eqnarray}
i.e. that $(M, g)$ is $\C^0$-asymptotic to Schwarzschild of mass $m>0$. The proofs of the statements in this appendix are straightforward extensions of those in \cite[Appendix A]{Eichmair-Metzger:2010} to higher dimensions, and we omit them. 
\\

\begin{lemma} \label{lem:area_growth_decay} Let $(M, g)$ be an initial
  data set. Let $\rho \geq 1$ and let $\Sigma \subset M$ be a closed hypersurface
  such that $\H^{n-1}_g(\Sigma \cap B_r \setminus B_\rho)
  \leq \Theta r^{n-1}$ for all $r \geq \rho$. Then the estimate
\begin{eqnarray*} 
\int_{\Sigma \setminus B_\rho} r^{-p} d \H^{n-1}_g \leq  \frac{p}{p -
(n-1)} \Theta \rho^{(n-1) - p}
\end{eqnarray*} holds for every $p>(n-1)$. 
\begin{proof} The proof uses the co-area formula as in \cite[p.
52]{Schoen-Yau:1979-pmt1}. 
\end{proof}
\end{lemma}

\begin{corollary} \label{cor:surface_comparison} Let $(M, g)$ be an
  initial data set. For all $\rho \geq 1$ and $\gamma \in (0, 1]$ and every closed
  hypersurface $\Sigma \subset M$ with $\H^{n-1}_g(\Sigma \cap B_r
  \setminus B_\rho) \leq \Theta r^{n-1}$ for all $r \geq \rho$ one has
\begin{eqnarray*}
  \int_{\Sigma \setminus B_\rho} r^{-(n-2+\gamma)} d \H^{n-1}_g
  \leq
  \rho^{-\beta} \H^{n-1}_g(\Sigma \setminus B_\rho)^{\frac{1-\gamma+\beta}{n-1}}
  \left(\frac{(n-2+\gamma) \Theta}{\beta}\right)^{\frac{n-2 +\gamma - \beta}{n-1}}
\end{eqnarray*} for every $\beta \in (0, n-2+\gamma)$. 
\end{corollary}

\begin{lemma} \label{lem:volume_comparison} Let $(M,g)$ be an initial
  data set for which the decay assumptions (\ref{eqn: decay
    assumptions integral decay Schwarzschild}) hold. There is a
  constant $C' \geq 1$ depending only on $C$ such that for every $\rho
  \geq 1$ and every bounded measurable subset $\Omega \subset M$ one
  has
  \begin{eqnarray*}
    |\CL^n_g(\Omega \setminus B_\rho) - \CL^n_{g_m}(\Omega
    \setminus B_\rho)|
    \leq
    C' \left( \frac{n-1+\gamma-\alpha}{\alpha - 1}\right)^{\frac{n - 1
        + \gamma - \alpha}{n}} \CL^n_{g}(\Omega \setminus
    B_\rho)^\frac{1-\gamma + \alpha}{n} \rho^{1- \alpha}
\end{eqnarray*} for every $\alpha
\in (1, n-1+\gamma)$. 

\begin{proof}
The volume elements differ by terms $O(r^{2-n-\gamma})$.
\end{proof}
\end{lemma}


\section{Hawking-mass} \label{sec:Hawking}

Let $g$ be a rotationally symmetric metric on $(a, b) \times
\S^{n-1}$. Given $r \in (a, b)$, let $A = A(r) := \H^{n-1}_g(\{r\} \times \S^{n-1})$ denote
the area of the coordinate sphere $\{r\} \times \S^{n-1}$, and $H =
H(r)$ its (scalar) mean curvature, computed as the tangential
divergence of the normal vector field in direction
$\partial_r$. Define the function
\begin{equation*}
  m (r) := \left(\frac{A}{\omega_{n-1}}\right)^{(n-2)/(n-1)} \left(1 - \frac{A^{2/(n-1)} H^2}{\omega_{n-1}^{2/(n-1)} (n-1)^2} \right). 
\end{equation*}
This expression appears in different but equivalent form in \cite[(13)]{Lee-Sormani:2011}. It is constructed so as to evaluate to the mass on the centered spheres in the Schwarzschild metric. In particular, it restricts to the usual Hawking mass in dimension $n=3$. 

\begin{lemma} [\protect{Cf. \cite[Section 2]{Lee-Sormani:2011}}] Assume that the scalar curvature of $g$ is non-negative, and that $r \to A(r)$ is non-decreasing. Then $m(r)$ is a non-decreasing function. If $m = m(c) = m(d)$ for some $c, d \in (a, b)$ with $c <d$, then $([c, d] \times \S^{n-1}, g)$ is isometric through a rotationally invariant map to $(\{x \in \R^n : c' \leq |x| \leq d'\}, (1 + \frac{m}{2 |x|^{n-2}})^{\frac{4}{n-2}} \delta_{ij})$ for some $0 < c' < d'$ such that $1 + \frac{m}{2 (c')^{n-2}} > 0$. \end{lemma}


\section{Standard Formulae} \label{sec:standard}

We collect several basic facts from Riemannian geometry, for ease of reference and to set forth the sign conventions that are used throughout the paper. 

We begin with our conventions for the Riemann curvature tensor. Let $X, Y, Z, W$ be vector fields on a Riemannian manifold $(M, g)$. Let $\nabla$ denote the Levi-Civita connection associated with $g$. Then $\Rm (X, Y, Z,  W) = g(\nabla_X (\nabla_Y Z) - \nabla_Y (\nabla_X Z) - \nabla_{[X, Y]} Z, W)$. The Ricci curvature is given by $\Ric (X, Y) := \text{trace}_{g} \Rm (\cdot, X, Y, \cdot)$. The scalar curvature is given by $\Scal := \text{trace}_g \Ric (\cdot, \cdot)$. 

\begin{lemma} [Kulkarni--Nomizu product] \label{lem:kulkarni}
Let $a_{ij}, b_{ij}$ be two symmetric $(0, 2)$ tensors. Then the $(0, 4)$ tensor 
$c_{{ijkl}} := (a \odot b)_{ijkl} =  a_{jk} b_{il} + a_{il} b_{jk} - a_{ik} b_{jl} -
a_{jl} b_{ik}$ has the symmetries of the
Riemann curvature tensor, i.e. $c_{ijkl} = - c_{jikl}$ and $c_{ijkl} =
c_{klij}$. If $(M, g)$ be a Riemannian manifold and if $Rm$ is its Riemann curvature tensor, then $\Rm = \frac{\overset{\circ}{\Ric}
\odot g}{n-2} + \frac{\Scal g \odot g}{2 n (n-1)} + \text{W}$, where $\overset{\circ}{\Ric}:= \Ric - \frac{\Scal}{n}g$ is the trace free part of the Ricci tensor, and where $\text{W}$ is the Weyl curvature. 
\end {lemma}

\begin{lemma} [Codazzi and Gauss equations] Let $\Sigma $ be a hypersurface
in a Riemannian manifold $(M, g)$, let $p \in \Sigma$, and let $\{e_1,
\ldots, e_{n-1}, \nu\}$ be a local orthonormal frame of $T M$ near $p$ such
that $\nu$ restricts to a unit normal vector field along $\Sigma$. We denote by $\bar g_{ij} = g(e_i, e_j)$ the induced metric on $\Sigma$, and
by $h_{ij} := g (\nabla_{e_i} \nu, e_j)$ the components of the second
fundamental form of $\Sigma$ with respect to $\nu$. Let $\bar g^{ij} h_{ij} = H$ be the scalar mean curvature of $\Sigma$. Then $\overline \nabla_{k} h_{ij} - \overline \nabla_i h_{kj} = \Rm_{k i \nu
j}$, where $\overline \nabla$ denotes covariant differentiation with respect to $\overline g$. We have that $\overline \Rm_{ijkl} = \Rm_{ijkl} + h_{il} h_{jk} - h_{ik} h_{jl}$.       
\end{lemma}


\section{The geometry of the spatial Schwarzschild metric} \label{sec:geometrySchwarzschild}

Consider the $n$-dimensional spatial Schwarzschild Riemannian manifold of mass $m>0$, $$\left(\R^n \setminus \{0\}, g_m := \left(1 + \frac{m}{2 r^{n-2}}\right)^{\frac{4}{n-2}} \sum_{i=1}^n d x_i^2\right),$$ where $r = |x|$. Given $r>0$, we will denote the centered coordinate sphere $\{ x \in \R^n : |x| = r\}$ in this coordinate system by $S_r$. The sphere $S_{r_h}$ with $r_h= \left(\frac{m}{2}\right)^{1/(n-2)}$ is called the horizon. We record the following properties of this geometry; our sign conventions here are those of Appendix \ref{sec:standard}. 
\begin{enumerate} [(a)]
\item The inversion $x \to r_h^2\frac{x}{|x|^2}$ induces a reflection symmetry of $g_m$ across the horizon. 
\item The $g_m$-area of $S_r$ is $\phi_m^{\frac{2(n-1)}{n-2}} r^{n-1} \omega_{n-1}$. 
\item The $g_m$-mean curvature with respect to the unit normal in direction of $\partial_r$ of $S_r$ equals $\phi_m^{-n/(n-2)}(1-\frac{m}{2r^{n-2}}) \frac{n-1}{r}$. The horizon $S_{r_h}$ is a minimal surface, and the mean curvature of the spheres $S_r$ for $r > r_h$ is positive.
\item The conformal factor $\phi_m := 1 + \frac{m}{2r^{n-2}}$ is harmonic with respect to the Euclidean metric $\sum_{i=1}^n d x_i^2$. The scalar and the Weyl curvature of $g_m$ vanish.
\item $$ \Ric_{g_m} = \frac{(n-2) m} {r^n \phi_m^{2n/(n-2)}} \left(g_m - n \phi_m^{4/(n-2)} dr \otimes dr \right)$$
\item $$\Riem_{g_m} = \frac{m} {r^n \phi_m^{2n/(n-2)}} \left( g_m \odot g_m - n \phi_m^{4/(n-2)} (dr \otimes dr) \odot g_m \right)$$
\end{enumerate}


\section{Regularity of isoperimetric regions and the behavior of minimizing sequences} \label{sec:regularity}

The regularity of isoperimetric regions in complete Riemannian manifolds is that of area minimizing boundaries (see \cite{Morgan:2003, Ritore-Rosales:2004, Ros:2005} and the references therein): 
\begin{proposition} Let $\Omega$ be an isoperimetric region in $(M, g)$. Its reduced boundary $\partial^*\Omega$ is a smooth hypersurface away from a singular set of Hausdorff dimension $\leq n-8$.
\end{proposition}

The following technical lemma, which is needed to check the hypotheses of Theorem \ref{thm:effective_volume_comparison}, follows from explicit comparison:
\begin{lemma} [\protect{Cf. \cite[Lemma 4.3]{Eichmair-Metzger:2010}}] \label{lem:quadraticareagrowthisoperimetric}
Let $(M, g)$ be an initial data set. There exists a constant $\Theta > 0$ so that for every isoperimetric region $\Omega$ with $\CL^n_g(\Omega) \geq 1$ one has that $\H^{n-1}_g(\partial \Omega \cap B_r) \leq \Theta r^{n-1}$ for all $r \geq 1$, and that $\H^{n-1}_g (\partial \Omega)^{\frac{1}{n-1}} \CL^n_g(\Omega)^{- \frac{1}{n}} \leq \Theta$. 
\end{lemma}

The following proposition characterizes the behavior of minimizing sequences for the isoperimetric problem (\ref{eqn:isoperimetric_area_function}) in initial data sets. It is a slight refinement of \cite[Theorem 2.1]{Ritore-Rosales:2004}: 
\begin{proposition} [Cf. \protect{\cite[Proposition 4.2]{Eichmair-Metzger:2010}}]\label{prop:cut_and_paste_for_minimizing_sequence} 
Given $V > 0$ there exists an isoperimetric region $\Omega \subset  M$ -- which may be empty --  and a sequence of coordinate balls $B(p_i, r_i)$ with $p_i \to \infty$ and $0 \leq r_i \to r \in [0, \infty)$ as $i \to \infty$ such that $\CL^n_g(\Omega) + \CL^n_g(B(p_i, r_i)) = V$ and such that $\H^{n-1}_g(\partial \Omega) + \H^{n-1}_g(\partial B(p_i, r_i)) \to A_g(V)$. If $r>0$ and $\CL^n_g(\Omega) >0$, then the mean curvature of $\partial \Omega$ equals $\frac{n-1}{r}$. 
\end{proposition}

The following lemma is standard, cf. \cite[Lemma 2.4]{Eichmair-Metzger:2010}. 
\begin{lemma} \label{lem:crudeisoperimetricinequality}
  Let $(M, g)$ be an initial data set. There exists a constant $C>0$ depending only on
  $( M,  g)$ such that
  \begin{equation*}
    \left( \int_{ M}
      |f|^{\frac{n}{n-1}} d \CL^n_{ g}\right)^{\frac{n-1}{n}} \leq C
    \int_{ M} |d f|_g d \CL^n_{ g} \text{ for all } f \in
    \C_c^1(M).
  \end{equation*}
  For any bounded Borel set $\Omega \subset  M$ with finite perimeter one
  has that
  \begin{equation*}
    \CL^n_{ g} (\Omega)^{\frac{n-1}{n}} \leq C \H^{n-1}_{
      g}(\partial^* \Omega).
  \end{equation*}
\end{lemma}


\section{An alternative expression for the center of mass}
The following lemma is an extension of \cite[Lemma 2.1]{Huang:2010}. Rather than applying a density theorem as in \cite{Huang:2010}, our proof below relies on an elementary integration by parts, cf. the papers \cite{Ma:2010, Ma:2012}  by S. Ma. 

\begin{lemma} [\protect{Cf. \cite[Lemma 2.1]{Huang:2010}}]
  \label{lemma:alternative_center}
  Let $g_{ij}$ be a metric on $\R^n$ that is $\C^2$-asymptotic to
  Schwarzschild of mass $m>0$ and asymptotically even at rate
  $\gamma \in (0, 1]$. Then for all $c>0$ and $\gamma_1 \in(0,1]$ there
  exists $c'>0$ such that for all $p\in\R^n$ with
  $|p| \leq c r^{1-\gamma_1}$ and $r \geq 1$ we have that
  \begin{multline*}
    \left|
      \int_{S_r(p)} (x_l-p_l) \big(H^{S_r(p)} - \frac{n-1}{r}\big)
      d\H_\delta^{n-1}
      - m(n-1) \omega_{n-1}(p_l - \C_l) 
    \right|
    \\
    \leq c' (r^{1-\gamma - \gamma_1} + r^{- \min\{\gamma,\gamma_1\}}).
  \end{multline*}
  Here, $H^{S_r(p)}$ denotes the mean curvature of $S_r(p)$ with
  respect to $g$.
\begin{proof}
  Throughout the proof we will sum over repeated indices. Let $h_{ij}
  := g_{ij} - \delta_{ij}$ and let $\rho = \frac{x-p}{|x-p|}$. Then
  \begin{equation}
    \label{eq:H_expansion}
    H^{S_r(p)} - \tfrac{n-1}{r}
    =
    \frac12 h_{ij,k} \rho_i \rho_j \rho_k
    + \frac12 h_{ii,j} \rho_j
    - h_{ij,i} \rho_j
    + \frac{n+1}{2} \frac{h_{ij}}{r} \rho_i \rho_j
    -\frac{h_{ii}}{r} 
    + E
  \end{equation}
  where $E$ is an error term with $|E(x)| \leq c' |x-p|^{3-2n}$, uniformly for $p$
  such that $2 |p| \leq r$ when $r$ is large. This follows from a
  calculation exactly as in the case $n=3$, cf. \cite[Lemma
  2.1]{Huang:2010}. Moreover, we have that
  \begin{eqnarray} \label{eqn:Eeven}
    |E(x) - E(2p-x)| \leq  c' |x-p|^{3-2n-\min\{\gamma,\gamma_1\}} \\ \nonumber \text{ for all } |p|  \leq c |x|^{1-\gamma_1} \text{ and } |x- p| \geq 1. 
  \end{eqnarray}
  We claim that for each $l \in \{1,\dots,n\}$,
  \begin{multline}
    \label{eq:integral_identity}
    \frac{1}{2} \int_{S_r(p)} (x_l - p_l) h_{ij,k} \rho_i
    \rho_j \rho_k \dmu_\delta
    \\
    =
    \frac12 \int_{S_r(p)}
    h_{il}\rho_i
      + (x_l - p_l) \left(
        \frac{h_{ii}}{r}
      - (n+1) \frac{h_{ij}}{r} \rho_i\rho_j
      + h_{ij,j}\rho_i \right)
    \dmu_\delta.
  \end{multline}
  To see this, define the vector field $X_{(l)} := (x_l -p_l) h_{ij}
  \rho_i \del_j$ and note that
  \begin{equation*}
    \int_{S_r(p)} \operatorname{div}^{S_r(p)}_\delta X_{(l)} \dmu_\delta
    =
    \int_{S_r(p)} H^{S_r(p)}_\delta \delta( X, \rho) \dmu_\delta.
  \end{equation*}
  Using that $H^{S_r(p)}_\delta = \frac{n-1}{r}$ and that
  \begin{equation*}
    \begin{split}
      \operatorname{div}^{S_r(p)}_\delta X_{(l)}
      &= 
      (\delta_{jk} - \rho_j\rho_k) \del_k X_{(l)}^j
      \\
      &=
      h_{i l}\rho_i
      + (x_l - p_l) \left(
      \frac{h_{ii}}{r}
      - 2\frac{h_{ij}}{r} \rho_i\rho_j
      + h_{ij,j}\rho_i
      - h_{ij,k}\rho_i \rho_j \rho_k \right)
    \end{split}
  \end{equation*}
  we obtain (\ref{eq:integral_identity}). Multiply
  \eqref{eq:H_expansion} by $(x_l - p_l)$ and integrate over
  $S_r(p)$. Using \eqref{eq:integral_identity} we arrive at
  \begin{multline} \label{eqn:Eintegral}
      \int_{S_r(p)} (x_l - p_l) (H^{S_r(p)} - \tfrac{n-1}{r} )
    \dmu_\delta
    \\
    =
    \frac12 \int_{S_r(p)} (x_l - p_l)(h_{ii,j} -
    h_{ij,i})\rho_j
    + (h_{i l} \rho_i - h_{ii} \rho_l)
    \dmu_\delta
    +
    \int_{S_r(p)} (x_l - p_l) E \dmu_\delta.
  \end{multline}
  Using \eqref{eqn:Eeven} we see that the last term has
  order $O(r^{3-n-\min\{\gamma,\gamma_1\}})$. 
  To analyze the first term, let
  \begin{equation*}
    Y_{(l)}
    :=
    (x_l( g_{ij,i} - g_{ii,j}) - (g_{jl} - g_{ii}\delta_{lj}))\partial_j
  \end{equation*}
  so that
  \begin{equation*}
    \C_l = \frac{1}{2 m (n-1) \omega_{n-1}} \lim_{R\to\infty}
    \int_{S_R(0)} Y_{(l)}^j \frac{x_j}{R} d\H_\delta^{n-1}. 
  \end{equation*}
  Note that $\div_\delta Y_{(l)} = x^l (\Scal_g + \partial g * \partial g)$. It follows that $| \div_\delta Y_{(l)} | \leq c r^{1-n-\gamma}$ and $|\div_\delta Y_{(l)} (x) + \div_\delta Y_{(l)} (-x) | \leq c r^{-n-\gamma}$.
  
  Let $R > r + |p|$. Then 
    \begin{eqnarray}
    && \int_{B_R(0) \setminus B_r(p)} \div_\delta Y_{(l)} d\CL^n_\delta - \int_{B_r(p) \setminus B_r(-p)} \div_\delta Y_{(l)} d \CL^n_\delta \\
    &=& \int_{B_R(0) \setminus (B_r(p) \cap B_r(-p))} \div_\delta Y_{(l)} d \CL^n_\delta  \nonumber  \\ 
    &=& \frac{1}{2} \int_{B_R(0) \setminus (B_r(p) \cap B_r(-p))} \div_\delta Y_{(l)} (x) + \div_\delta Y_{(l)}(-x) d \CL^n_\delta (x). \nonumber
    \end{eqnarray}
    Hence 
    \begin{eqnarray}
  && \left| \int_{B_R(0) \setminus B_r(p)} \div_\delta Y_{(l)} d\CL^n_\delta \right| \\ &\leq& c' \int_{B_R (0) \setminus B_\frac{r}{2} (0)} r^{-n-\gamma} d \CL^n_\delta +  \left| \int_{B_r(p) \setminus B_r(-p)} \div_\delta Y_{(l)} d \CL^n_\delta \right| \nonumber \\ &\leq& 
  c' ( r^{-\gamma} \nonumber + r^{1-\gamma-\gamma_1})
     \end{eqnarray}
     where we have used the estimate  for the odd part of $\div_\delta Y_{(l)}$ in the first inequality, and the estimate for $\div_\delta Y_{(l)}$ and that $\CL^n_\delta (B_r(p) \setminus B_r(-p)) \leq c' r^{n-\gamma_1}$ in the second inequality. We emphasize that the right hand side is independent of $R$. Using the divergence theorem, we have that   \begin{multline}
    \label{eq:hcalc_divergence_identity}
    \int_{B_R(0) \setminus B_r(p)} \div_\delta Y_{(l)} d\CL^n_\delta
    \\=
    \int_{S_R(0)} Y_{(l)}^j \frac{x_j}{R} d\H^{n-1}_\delta
    -
    \int_{S_r(p)} x_l(h_{ij,i} - h_{ii,j})\rho_j + (h_{ii} \rho_l -
    h_{il}\rho_i) d\H^{n-1}_\delta.
  \end{multline}
Letting
  $R\to\infty$ in \eqref{eq:hcalc_divergence_identity} we obtain
  \begin{multline}
    \label{eq:hcalc_xterm}
    \left| \int_{S_r(p)} x_l(h_{ij,i} - h_{ii,j})\rho_j + (h_{ii} \rho_l -
      h_{il}\rho_i) d\H^{n-1}_\delta - 2 m (n-1) \omega_{n-1} \C_l
    \right|
    \\
    \leq
    c' (r^{1-\gamma -\gamma_1} + r^{-\gamma}).
  \end{multline}
  We claim that 
  \begin{equation}
    \label{eq:hcalc_pterm}
    \left| \int_{S_r(p)} p_l(h_{ij,i} - h_{ii,j})\rho_j
      d\H^{n-1}_\delta 
      - 2 m (n-1) \omega_{n-1} p_l
    \right|
    \leq
    Cc' r^{1 - \gamma - \gamma_1}.
  \end{equation}
 To see this, define the vector field $Y := (h_{ij,i} - h_{ii,j}) \partial_j$, note that $\div_\delta Y = R_g + \partial g * \partial g = O(r^{-n-\gamma})$, and that 
   \begin{equation*}
    \left| \int_{S_r(p)} (h_{ij,i} - h_{ii,j})\rho_j
      d\H^{n-1}_\delta 
      - 2 m (n-1) \omega_{n-1}
    \right|
    \leq
    \left|\int_{\R^n \setminus B_r(p)} \div_\delta Y d \CL^n_\delta \right| \leq c' r^{-\gamma}.
  \end{equation*}
 The lemma follows combining \eqref{eqn:Eintegral}, \eqref{eq:hcalc_xterm} and \eqref{eq:hcalc_pterm}. 
\end{proof}
\end{lemma}
  
  
\section{Appoximation by Spheres}
 
 \begin{lemma}  [\protect{Cf. \cite[Lemma 4.8]{Huang:2010} and \cite[Proposition 2.1]{Huisken-Yau:1996}}] \label{lem:sphere_approximation}
There exist $\delta, c >0$ depending only on $n$ so that the following holds: Let $\Sigma \subset \R^n$ be a closed hypersurface. If for some constant $\bar H >0$ one has that $\sup_{\Sigma} |\tf| + \sup_\Sigma |H - \bar H|\leq \delta \bar H$, then $\Sigma$ is strictly convex, and there exist $r \in (\frac{1}{2} \frac{n-1}{\bar H}, 2 \frac{n-1}{\bar H})$, $p \in \R^n$, and a function $v \in \C^2(S_r(p))$ such that $\Sigma = \{ x + v(x) \frac{x - p}{|x - p|} : x \in S_r(p)\}$ and 
  \begin{equation*}
   \sup_{S_r(p)} r^{-1} |v| + |Dv| + r|D^2 v|  \leq  c r (\sup_\Sigma |\tf| + \sup_\Sigma |H-\bar H|).
  \end{equation*}
\end{lemma}


\section{Overview of results on isoperimetric regions} \label{sec:overview}

Our intention in this section is three-fold: First, to give a complete account of all closed Riemannian manifolds whose isoperimetric regions are fully or largely characterized; second, to describe briefly all techniques and developments in the theory of isoperimetry that appear to us relevant in the context of this paper; and third, to provide the reader with an introduction to the rich literature on this subject.  


\subsection{Monographs and surveys}

R. Osserman's article \cite{Osserman:1978} surveys the classical literature on the isoperimetric problem. We point out in particular the discussion in Section $2$, which highlights the difference between characterizing critical and stable critical surfaces for the isoperimetric problem and establishing a sharp isoperimetric inequality, as well as the discussion in Section $4$ on results and conjectures related to the validity of the planar Euclidean isoperimetric inequality $L^2 - 4\pi A \geq 0$ on Riemannian surfaces with non-positive curvature. R. Osserman's article \cite{Osserman:1979} gives several effective (``Bonnesen-style") isoperimetric inequalities on Riemannian surfaces. The proofs depend on the Gauss-Bonnet theorem and F. Fiala's method of ``interior parallels", cf. the historical discussion in Section II. Section III.C contains some some extensions to higher dimension. The extensive monograph \cite{Burago-Zalgaller:1988} by Y. D. Burago and V. A. Zalgaller emphasizes the rich connection with convex and integral geometry and contains many interesting historical references. The more recent survey articles \cite{Ros:2005} by A. Ros and \cite{Ritore:2010} by M. Ritor\'e contain a wealth of additional material and up-to-date references. 


\subsection {Classical isoperimetric inequality} The sharp isoperimetric inequality in the simply connected constant curvature spaces $\R^n, \mathbb{S}^n$, and $\mathbb{H}^n$ have been established rigorously in all dimensions in a series of papers by E. Schmidt in the 1940's, cf. the Historical Remarks $10.4$ as well as Sections $8$-$10$ in \cite{Burago-Zalgaller:1988}. The isoperimetric regions are exactly the geodesic balls.  

\subsection{The case of surfaces}
The isoperimetric regions of certain rotationally symmetric surfaces have been completely characterized, using curve shortening flow \cite{Benjamini-Cao:1996,Topping:1998}, parallel surfaces techniques \cite{Fiala:1941,Pansu:1998,Topping:1999}, and by analysis of curves of constant geodesic curvature \cite{Morgan-Hutchings-Howards:2000,Ritore:2001A,Canete:2007,Canete-Ritore:2008}. The introduction of the recent article \cite{Canete-Ritore:2008} by A. Ca\~ nete and M. Ritor\'e contains a thorough overview of these results.

M. Ritor\'e \cite{Ritore:2001B} has shown that solutions of the isoperimetric problem exist for every volume in complete Riemannian planes with non-negative curvature. Conversely, in \cite{Ritore:2001A}, M. Ritor\'e  gives examples of complete rotationally symmetric planes in which no optimizers for the isoperimetric problem exist for any volume. 


\subsection{Symmetrization techniques} We refer the reader to Sections 1.3 and 3.2 in \cite{Ros:2005} and to Section 1.3 \cite{Ritore:2010} for brief descriptions of the symmetrization techniques by J. Steiner and H. Schwarz \cite{Schwarz:1890} as well as W.-T Hsiang and W.-Y. Hsiang \cite{Hsiang-Hsiang:1989}. 

In \cite{Hsiang-Hsiang:1989}, W.-T. Hsiang and W.-Y. Hsiang apply their symmetrization technique to reduce the study of isoperimetric regions in $\R^n \times \mathbb{H}^{m}$ and in $\mathbb{H}^m \times \mathbb{H}^n$ to an ODE analysis of curves in the plane. In $\R \times \mathbb{H}^2$, the solutions are completely characterized.

R. Pedrosa and M. Ritor\'e \cite{Pedrosa-Ritore:1999} have characterized the isoperimetric domains of $\S^1 \times \S^2$ and $\S^1 \times \mathbb{H}^2$ as well as of $\S^1 \times \R^{n-1}$ when $3 \leq n \leq 8$, using symmetrization as in \cite{Hsiang-Hsiang:1989} and ODE and stability analysis. 

R. Pedrosa \cite{Pedrosa:2004} has used spherical symmetrization as in \cite{Hsiang-Hsiang:1989} to show that the isoperimetric regions in $\R \times \mathbb{S}^{n-1}$ are connected and smooth, and either topological balls or cylindrical of the form $(a,b)\times \mathbb{S}^{n-1}$. In $\R \times \mathbb{S}^2$, the author has obtained an explicit description of the isoperimetric regions. 


\subsection{Small isoperimetric regions in Riemannian manifolds}

D. Johnson and F. Morgan \cite{Johnson-Morgan:2000} have shown that isoperimetric regions of small volume in closed Riemannian manifolds are perturbations of small geodesic balls. An alternative argument that applies in dimension $n=3$ is given in Theorem 18 of \cite{Ros:2005}. For the relationship between small isoperimetric regions and scalar curvature we refer to the work of R. Ye \cite{Ye:1991}, P. Pansu \cite{Pansu:1998}, O. Druet \cite{Druet:2002a, Druet:2002b}, and S. Nardulli \cite{Nardulli:2009}.   


\subsection{Classifying stable constant mean curvature surfaces}
Using a particular choice of test function in the stability inequality, J. L. Barbosa, M. DoCarmo \cite{Barbosa-DoCarmo:1984}, and J. L. Barbosa, M. DoCarmo, and J. Eschenburg \cite{Barbosa-DoCarmo-Eschenburg:1988} have shown that in the simply connected space forms, every closed volume preserving stable constant mean curvature hypersurface is a geodesic sphere. 

M. Ritor\'e and A. Ros \cite{Ritore-Ros:1992} have characterized the isoperimetric regions in $\mathbb {RP}^3$ and $\R^3/S_\theta$, where $S_\theta$ is a subgroup of $O(3)$ generated by a translation or a screw motion, using the stability inequality in several different ways.  In \cite{Ritore-Ros:1996}, the same authors characterize the isoperimetric regions of most products $\mathbb{T}^2 \times \R$ where $\mathbb{T}^2$ is a flat $2$-torus.  (``Most" means all those from a compact subset in the non-compact moduli space of such manifolds.) A full characterization of the isoperimetric regions of $T^2 \times \R$, where $T$ is a flat torus with injectivity radius $1$ and area greater than a certain $\epsilon>0$, has been given by M. Ritor\'e in \cite{Ritore:1997}. 


\subsection{Isoperimetric comparison}

We point out the Levy-Gromov comparison theorem for the isoperimetric profile for closed Riemannian manifolds whose Ricci curvature is bounded below by that of the sphere, cf. Theorem 19 in \cite{Ros:2005}. B. Kleiner \cite{Kleiner:1992} has proven a sharp isoperimetric comparison result for three dimensional Hadamard manifolds. Alternative proofs of B. Kleiner's result have been given by M. Ritor\'e \cite{Ritore:2005} and by F. Schulze \cite{Schulze:2008}; these proofs are surveyed in \cite[Sections 3.2 and 3.3]{Ritore:2010}.  

H. Bray \cite{Bray:1998} has characterized the isoperimetric regions homologous to the horizon in the spatial Schwarzschild manifold. His method has been extended by H. Bray and F. Morgan \cite{Bray-Morgan:2002} to general spherically symmetric manifolds satisfying certain conditions. J. Corvino, A. Gerek, M. Greenberg, and B. Krummel
\cite{Corvino-Gerek-Greenberg-Krummel:2007} have applied the methods of \cite{Bray:1998, Bray-Morgan:2002} to characterize the isoperimetric regions in the spatial Reissner–Nordstrom and Schwarzschild anti de Sitter manifolds. In \cite{Brendle-Eichmair:2011}, S. Brendle and the first author use the results in \cite{Eichmair-Metzger:2010, Brendle:2011} to characterize the isoperimetric regions in the ``doubled" Schwarzschild manifold, complementing the results of H. Bray in \cite{Bray:1998}. 


\subsection {Effective isoperimetric inequalities} \label{sec:effectiveiso}
 
 In \cite{Fusco-Maggi-Pratelli:2008}, N. Fusco, F. Maggi, and A. Pratelli have given an effective isoperimetric inequality (``Bonnesen-style" as coined by R. Osserman \cite{Osserman:1979}) for sets in $\R^n$. Their result is sharp in a sense that the authors make precise. Their proof is based on Schwarz-Steiner symmetrization. See also the paper \cite{Figalli-Maggi-Pratelli:2010} by A. Figalli, F. Maggi, and A. Pratelli and the paper \cite{Cicalese-Leonardi:2010} by M. Cicalese and P. Leonardi for alternative proofs based respectively on optimal transport and explicit minimization. 


\bibliographystyle{amsplain}
\bibliography{references}

\end{document}